\RequirePackage{fix-cm}
\documentclass[smallcondensed]{svjour3}     	% onecolumn (ditto)
\smartqed  % flush right qed marks, e.g. at end of proof
\usepackage{graphicx}
\usepackage{amsmath}
\usepackage{amssymb}
\usepackage{color}
\usepackage[caption=false,font=footnotesize]{subfig}
\usepackage{multirow}
\usepackage{cite}
\usepackage{epstopdf}
\newcommand{\I}{\mathrm{i}}

\newcommand{\D}{\mathrm{d}}
\newcommand{\U}{\mathbf{u}}
\newcommand{\V}{\mathbf{v}}
\newcommand{\x}{\mathbf{x}}
\newcommand{\im}{\mathrm{i}}

\definecolor{grey}{rgb}{0.53,0.54,0.52}
\definecolor{darkred}{rgb}{0.6,0,0}
\definecolor{lightgrey}{gray}{0.85}

\begin{document}

\title{Stable difference methods for block-oriented adaptive grids%\thanks{Grants or other notes
%about the article that should go on the front page should be
%placed here. General acknowledgments should be placed at the end of the article.}
}
%\subtitle{Do you have a subtitle?\\ If so, write it here}

%\titlerunning{Short form of title}        % if too long for running head

\author{Anna Nissen \and Katharina Kormann\and Magnus Grandin  \and Kristoffer Virta}

\authorrunning{Nissen, Kormann, Grandin, Virta} % if too long for running head

\institute{A. Nissen \at
              Dept. of Mathematics, University of Bergen, Bergen, Norway \\
%              Tel.: +123-45-678910\\
%              Fax: +123-45-678910\\
              \email{anna.nissen@math.uib.no}           %  \\
%             \emph{Present address:} of F. Author  %  if needed
           \and
           K. Kormann \at
%              Dept. of Information Technology, Uppsala University, Uppsala, Sweden \\
              Zentrum Mathematik, Technische Universit{\"a}t M{\"u}nchen, Munich, Germany \\
%              Tel.: +123-45-678910\\
%              Fax: +123-45-678910\\
             \email{katharina.kormann@tum.de}           %  \\
          \and
           M. Grandin \and K. Virta \at 
              Dept. of Information Technology, Uppsala University, Uppsala, Sweden \\
%              Tel.: +123-45-678910\\
%              Fax: +123-45-678910\\
             \email{\{magnus.grandin, kristoffer.virta\}@it.uu.se}           %  \\
%		\and
%             K. Virta \at
%            Dept. of Information Technology, Uppsala University, Uppsala, Sweden \\
%%              Tel.: +123-45-678910\\
%%              Fax: +123-45-678910\\
%              \email{kristoffer.virta@it.uu.se}           %  \\
}

%%\date{Received: date / Accepted: date}
% The correct dates will be entered by the editor
\date{}

\maketitle

\begin{abstract}
%Insert your abstract here. Include keywords, PACS and mathematical
%subject classification numbers as needed.
In this paper, we present a block-oriented scheme for adaptive mesh refinement based on summation-by-parts (SBP) finite difference methods and simultaneous-approximation-term (SAT) interface treatment. Since the order of accuracy at SBP-SAT grid interfaces is lower compared to that of the interior stencils, we strive at using the interior stencils across block-boundaries whenever possible. We devise a stable treatment of SBP-FD junction points, i.e. points where interfaces with different boundary treatment meet. This leads to stable discretizations for more flexible grid configurations within the SBP-SAT framework, with a reduced number of SBP-SAT interfaces. Both first and second derivatives are considered in the analysis. Even though the stencil order is locally reduced close to numerical interfaces and corner points,  numerical simulations show that the locally reduced accuracy does not severely reduce the accuracy of the time propagated numerical solution. Moreover, we explain how to organize the grid and how to automatically adapt the mesh, aiming at problems of many variables. Examples of adaptive grids are demonstrated for the simulation of the time-dependent Schr{\"o}dinger equation and for the advection equation.

%%%%%%%%%

%In this paper, we present a block-structured adaptive mesh refinement scheme based on summation-by-parts (SBP) finite differences and simultaneous-approximation-term (SAT) interface treatment. Since the order of accuracy at SBP-SAT grid interfaces is lower compared to that of the interior stencil, we reduce the number of SBP-SAT interfaces whenever possible. Therefore, we devise a stable treatment of T-junction points where interfaces with different boundary treatment meet. Both first and second derivatives are considered. A numerical convergence study confirms our theoretical findings. Moreover, we explain how to organize the grid targeting at problem of many variables and how to automatically adapt the mesh and show examples on adaptive grids for the simulation of the time-dependent Schr{\"o}dinger equation.
%
%We devise a stable numerical treatment for multi-configurational grids that include grid corners, 

%\keywords{summation-by-parts \and block-adaptive grid \and SBP-FD junction \and time-dependent Schr{\"o}dinger equation \and advection equation} % TODO: Insert keywords
% \PACS{PACS code1 \and PACS code2 \and more}
% \subclass{MSC code1 \and MSC code2 \and more}
\end{abstract}

%\section{Introduction}
%\label{intro}
%Your text comes here. Separate text sections with
%\section{Section title}
%\label{sec:1}
%Text with citations \cite{RefB} and \cite{RefJ}.
%\subsection{Subsection title}
%\label{sec:2}
%as required. Don't forget to give each section
%and subsection a unique label (see Sect.~\ref{sec:1}).
%\paragraph{Paragraph headings} Use paragraph headings as needed.
%\begin{equation}
%a^2+b^2=c^2
%\end{equation}
%
%% For one-column wide figures use
%\begin{figure}
%% Use the relevant command to insert your figure file.
%% For example, with the graphicx package use
%  \includegraphics{example.eps}
%% figure caption is below the figure
%\caption{Please write your figure caption here}
%\label{fig:1}       % Give a unique label
%\end{figure}
%%
%% For two-column wide figures use
%\begin{figure*}
%% Use the relevant command to insert your figure file.
%% For example, with the graphicx package use
%  \includegraphics[width=0.75\textwidth]{example.eps}
%% figure caption is below the figure
%\caption{Please write your figure caption here}
%\label{fig:2}       % Give a unique label
%\end{figure*}
%%
%% For tables use
%\begin{table}
%% table caption is above the table
%\caption{Please write your table caption here}
%\label{tab:1}       % Give a unique label
%% For LaTeX tables use
%\begin{tabular}{lll}
%\hline\noalign{\smallskip}
%first & second & third  \\
%\noalign{\smallskip}\hline\noalign{\smallskip}
%number & number & number \\
%number & number & number \\
%\noalign{\smallskip}\hline
%\end{tabular}
%\end{table}

\section{Introduction}

%Accurate numerical simulation of time-dependent phenomena in many spatial dimensions is a challenge in a wide range of application areas, for example quantum dynamics \cite{}, systems biology \cite{Ferm10}, plasma physics \cite{Holmstrom09}, and financial mathematics \cite{Linde09}. In particular, the number of grid points needed to represent a solution with high enough resolution often becomes prohibitively large, and high-order and adaptive schemes are an active area of research in these application fields.

%Something about adaptivity for problems with high dimensionality and different physical properties. 

Accurate numerical simulation of time-dependent phenomena in many spatial dimensions is a challenge in a wide range of application areas, for example quantum dynamics \cite{Meyer90} and systems biology \cite{Ferm10}. Other challenging areas are wave propagation problems with drastically varying physical features in different spatial regions, governed by e.g. the elastic wave equation~\cite{Appelo09}. Computational problems in higher dimensions and/or widely varying scales are demanding since the number of grid points required to represent a solution with high enough resolution often becomes prohibitively large. As a consequence, high-order and adaptive schemes are an active area of research in these application fields. The focus of our paper is on wave propagation problem described either by first derivatives or complex second derivatives. 
%In particular, the number of grid points needed to represent a solution with high enough resolution often becomes prohibitively large, and high-order and adaptive schemes are an active area of research in these application fields. 

When structured grids can be used, finite difference methods allow for efficient implementation of high-order methods. Summation-by-parts (SBP) operators are finite difference operators with special boundary closure such that the discrete operators mimic properties of the continuous operators. An attractive feature of the SBP operators is that in combination with the simultaneous-approximation term (SAT) boundary treatment~\cite{Carpenter94}, this discretization often leads to time-stability in single- as well as in multiblock configurations~\cite{Mattsson10}. The combined scheme is referred to as SBP-SAT. Since the original development by Kreiss and Scherer \cite{Kreiss74} for first order derivatives, several contributions have been made to further develop SBP operators (see eg. \cite{Strand94,Mattsson2004}). The SBP-SAT framework has also been successfully applied to a large range of physical problems, see for example \cite{Svard08, Lindstrom10, Kozdon13}, resulting in robust discretizations.

%An attractive feature of the SBP operators is that in combination with the simultaneous-approximation term (forT) boundary treatment~\cite{Carpenter94}, this discretization often leads to time-stability in single- as well as in multiblock configurations~\cite{Mattsson10}. The combined scheme is referred to as SBP-SAT.

%TODO: Introducera SBP-SAT, referenser, single/multiblock
% Originally, the SAT treatment has been applied to physical boundaries. However, Mattsson and Carpenter \cite{Mattsson10} have recently shown that it is possible 

 %The SAT terms (or penalty terms) couple neighboring blocks by including a term in the difference operator that only affects the points closest to block boundaries. By choosing the penalty terms correctly, time-stablility of the grid block couplings can be shown via the energy method. Therefore, we believe that the SBP-SAT discretization will be a constituent in a robust framework for adaptive mesh refinement.
 
Originally, the SAT treatment has been applied to physical boundaries and to numerical interfaces for which the collocation points in neighboring blocks coincide. More recently, similar boundary treatments have been used to combine high-order finite difference operators with adaptive mesh refinement \cite{Mattsson10, Kramer09}. In this setting, the computational domain is decomposed into multiple blocks with different refinement level. On each block, the SBP technique is used for discretization.
On boundaries of patches with different refinement, so called nonconforming interfaces, one essentially has two choices to implement the coupling between the blocks. In a continuous formulation the stencil across block-boundaries is modified. In a discontinuous formulation each patch is discretized separately and penalty terms are added that couple the discretizations across the interfaces. In each case, the interface treatment gets increasingly complicated with the order of accuracy of the method.

Mattsson and Carpenter \cite{Mattsson10} have studied the case of a discontinuous formulation and derived interpolation operators and penalty terms also for high-order methods. However, they only treat interfaces between two blocks of different refinement levels. When studying a fully adaptive mesh we will also encounter corner points, i.e. points where two or more interfaces between blocks of different refinement level intersect. Treating the corner points in a stable way that does not reduce the overall accuracy of the scheme is a challenging task. Kramer and coworkers \cite{Kramer09} have studied a continuous formulation based on SBP. They showed how to, in a stable manner, handle edges as well as corners where patches of different refinement level meet. The stability comes at the cost of a lower accuracy compared to the order of accuracy of the interior stencil. 

%One main contribution of this article is to explain how to handle corner points in a stable way while keeping extra refinements of the mesh to a minimum. 
In this paper, we explain how to handle corner points in a stable way while keeping extra refinements of the mesh to a minimum. Contrary to the approach in \cite{Kramer09}, our study is based on a discontinuous formulation where the SBP-SAT interfaces are combined with the interpolation operators derived in \cite{Mattsson10}. The key idea of our approach is to allow for junctions where different types of block-boundary treatments are allowed. We design new, so called SBP-FD junction operators, that allow for a more flexible grid treatment within the SBP-SAT framework. Similarly as in \cite{Kramer09}, we sacrify some accuracy in order to obtain a stable discretization. 

Berger and Oliger \cite{BergerOliger84} developed an approach for structured adaptive mesh refinement (SAMR), for multiple component grids in a finite difference setting. In their original method, the refined regions can be arbitrarily placed and oriented with respect to underlying grid patches. Berger and Colella \cite{BergerColella89} modified the algorithm such that patches are restricted to be aligned with one another, which significantly simplifies the mesh organization. Although the Berger-Colella approach is an improvement of the original Berger-Oliger method perfomance-wise, it still suffers from overhead in the mesh adaptation step~\cite{Rantakokko09}. Moreover, possibly overlapping patches of arbitrary size and shape result in complex data dependencies in the grid hierarchy, which complicates matters related to parallelization and load balancing on large compute clusters~\cite{MacNeice00,Rantakokko09}. 

Block-oriented SAMR (also referred to as block-based or block-wise SAMR in the literature) has been proposed as an alternative to the aforementioned methods of mesh refinement. In this class of methods, the computational domain is decomposed into a hierarchy of non-overlapping grid blocks, where refinement is undertaken with respect to entire grid blocks only. In comparison with the method of Berger-Colella, the overhead associated with grid management is reduced and load balancing becomes a simpler and more straightforward task \cite{Rantakokko09}. There are two strategies available for block refinement. One strategy is to introduce more grid points in the blocks where refinement is needed, keeping the number of blocks constant. The other strategy is to keep the block size in terms of grid points constant and increase the number of blocks in regions where refinement is needed. In effect, a block that needs refinement is subdivided and replaced by a number of sub-blocks covering the exact same sub-region as the original block, resulting in a finer resolution in that region. We build our implementation and numerical techniques upon the latter scheme. Among the examples of scientific software packages that implement block-oriented AMR in this fashion are PARAMESH~\cite{MacNeice00} and Racoon~\cite{Dreher05}. Both of these packages implement a similar refinement scheme and grid organization, where blocks are always subdivided isotropically (i.e., uniformly in all dimensions). In our approach, we allow for \textit{anisotropic} refinement and let blocks be refined in one, a few or all dimensions as required to fulfill the refinement criteria. We expect the anisotropic grid refinement strategy to become increasingly important as we aim at tackling higher-dimensional problems, since this will keep the number of created blocks to a minimum and thereby reduce the overall memory requirements.

%PARAMESH \cite{MacNeice00} is an open-source tool based on block-wise SAMR with fixed size blocks that provides mesh organization for parallel use of serial code. The grid organization in our approach is similar to that of PARAMESH, with the distinction that PARAMESH uses isotropic grid refinement whereas our grid refinement strategy is to refine anisotropically. We expect that the anisotropic grid refinement strategy will be important in order to tackle high-dimensional problems.

In this paper, we consider the discretization of partial differential equations (PDE) with first and second order derivatives. Stability analysis and numerical convergence studies are provided for the advection equation and for the free particle Schr\"{o}dinger equation to illustrate the applicability of our approach to different classes of problems. As examples with adaptive refinement, we consider the quantum harmonic oscillator as well as the advection equation. 

The article is organized as follows. Section \ref{GridOrganisation} deals with the hierarchical grid structure. Section \ref{sec:second} introduces the discretization for a parabolic problem with second derivatives in space. The stable treatment of SBP-FD junctions where interfaces of different type meet is devised and the accuracy is discussed. A discretization of a hyperbolic equation with first derivatives on block-oriented grids is the subject of Section~\ref{sec:first}. An extensive numerical convergence study for different types of corner points occurring in a block-adaptive grid is provided in Section~\ref{sec:num_convergence}. Finally, we explain how to automatically adapt the mesh in Section~\ref{sec:adaptivity}. Numerical examples are presented for the example of the time-dependent Schr{\"o}dinger equation and for the advection equation. Section \ref{section7} gives concluding remarks and an outlook on future work.

\section{Grid organization}

\label{GridOrganisation}

We implement a conservative block-oriented refinement scheme that strives to minimize the fan-out of the mesh refinement and avoid wasting memory on unnecessarily fine grid blocks. Meshes are structured such that all blocks represent an equal number of grid points but due to varying refinement	they correspond to logical $d$-dimensional hyperrectangles ($d$-orthotopes) of different size. For simplicity of implementation, we restrict the refinement such that two adjoining blocks can differ in refinement ratio by at most a factor $2$ in each direction. %TODO: Add more motivation here.

Grid refinement is carried out block-by-block in an anisotropic manner (i.e., blocks do not have to be refined uniformly in all dimensions). This gives us the freedom to refine a block as needed in the dimensions where refinement is required while leaving the block coarser in the other dimensions. We expect this strategy to generate fewer grid blocks than isotropic refinement and therefore to be more memory efficient. In higher spatial dimensions, this gain will become increasingly significant, in particular if the solution is elongated or has oscillations in some dimensions only. % For an example of an anisotropically refined multiblock grid, see Fig.~\ref{subfig:block_grid}.

Upon refinement of a block, the block to be refined is halved in the desired dimension, generating two new blocks that are filled with intermediate data values such that the resulting spatial resolution in that dimension is twice that of the original block. In $d$ dimensions, this can be generalized to a hyperplane cutting through a $d$-orthotope, splitting it in two parts. A block can be as elongated as is needed, in principle without any restrictions on the ratio between the lengths of its edges. Depending on the properties of the numerical discretization method, however, such restrictions might arise for stability reasons. The methods we use in this paper do not have any such formal restrictions, but the required time step size will largely be affected by the smallest spatial step size. % More details/motivation needed here? Can we safely say that there is no restriction on the aspect ratio?

%\begin{figure}[!t]
%\subfloat[][]{\includegraphics[width=4.7cm]{kdtree}\label{subfig:kdtree}}
%\hspace{10pt}
%\subfloat[][]{\includegraphics[width=3.7cm]{block_grid}\label{subfig:block_grid}}
%\caption{Illustration of an anisotropically refined grid and the corresponding $kd$-tree that represents its hierarchical structure. The split dimension of each internal node is indicated by $x$/$y$ and the bit index of each generated child node is indicated by 0/1.}
%\label{fig:grid_organization}
%\centering
%\end{figure}

 %(1 page)

\section{Second derivatives}\label{sec:second}

We consider the Schr\"{o}dinger equation
\begin{equation}
  U_t = \im \Delta U - \im V U, \label{TDSE}
\end{equation}
with initial and boundary conditions. Here, $V$ denotes a spatially dependent potential operator. Since $V$ has no impact on the stability analysis (cf. \cite{Nissen11}), we consider the free particle case, for which $V=0$. We discretize the PDE based on the method of lines. As such, we first discretize in space using finite difference methods. The resulting system of ordinary differential equations is then propagated in time using an approach based on exponential integrators. The spatial and temporal discretization methods are described below. At the end of Sec.~\ref{sec:Accuracy}, we comment on the case of real second derivatives modeling diffusion.

\subsection{Spatial discretization}
The spatial discretization is carried out using SBP operators, a finite difference discretization with central difference stencils in the interior and one-sided stencils close to the block boundaries. Note that the boundary stencils are of lower order than the order of the interior scheme. In the multiblock structure proposed in Sec.~\ref{GridOrganisation}, various types of block boundaries need to be treated. Across boundaries between blocks of precisely the same refinement in all dimensions, central finite differences (c-FD) can be used. This is the simplest and most accurate way of discretizing over block boundaries. However, the arrangement of blocks in the vicinity of such an interface might require SAT couplings to be enforced across the boundary in order to preserve stability (cf. Sec.~\ref{sec:junction}). Grid blocks for which the refinement level along the interface is identical but the refinement level in the dimension perpendicular to the interface differs are coupled using pure SAT terms. If the refinement level along the interface differs, i.e. for nonconforming interfaces, the SAT terms are combined with interpolation and projection operators.

The coupling between nonconforming grid blocks is enforced by a combination of SAT terms and interpolation and projection operators constructed for SBP operators (cf.~\cite{Mattsson10}). In order to preserve the stability of the semi-discretization, proper coupling terms can be determined using the energy method. For a more precise description of the semi-discretization for the time-dependent Schr{\"o}dinger equation and the detailed form of the coupling terms leading to a stable semi-discretization on nonconforming grids, we refer to~\cite{Nissen11}. 

\begin{figure}[!t]
  \centering
  % \setlength{\unitlength}{0.6pt}
%  \subfloat[][]{
%    \begin{picture}(80,80)
%      \put(0,0){\line(1,0){80}}
%      \color{blue}\put(0,40){\line(1,0){80}}
%      \color{black}\put(0,80){\line(1,0){80}}
%      \put(0,0){\line(0,1){80}}
%      \color{red}\put(40,0){\line(0,1){80}}
%      \color{green}\put(40,40){\circle*{3}}
%      \color{blue}\put(60,0){\line(0,1){80}}
%      \color{black}\put(80,0){\line(0,1){80}}
%      \color{blue}\put(40,20){\line(1,0){40}}
%      \put(40,60){\line(1,0){40}}
%      \color{black}\put(5,5){1}
%      \put(5,45){2}
%      \put(45,5){3}
%      \put(65,5){4}
%      \put(45,25){5}
%      \put(65,25){6}
%      \put(45,45){7}
%      \put(65,45){8}
%      \put(45,65){9}
%      \put(65,65){10}
%    \end{picture}
%      \label{grid1}
%  }  
  \subfloat[][]{
    \begin{picture}(80,80)
     \put(0,0){\line(0,1){80}}
      \color{red}\put(39,0){\line(0,1){80}}
      \color{red}\put(41,0){\line(0,1){80}}
     \color{blue} \put(60,0){\line(0,1){39}}%\multiput(60,0)(0,5.25){8}{\line(0,1){3}}
      \color{black}\put(80,0){\line(0,1){80}}
     \color{blue} \put(41,20){\line(1,0){39}}%\multiput(40,20)(5.25,0){8}{\line(1,0){3}}
      \color{black}\put(0,0){\line(1,0){80}}
      \color{red}\put(0,39){\line(1,0){80}}
      \color{red}\put(0,41){\line(1,0){80}}
      \color{black}\put(0,80){\line(1,0){80}}
       \color{black}\put(40,40){\circle*{6}}
       \color{black} \put(5,5){3}
      \put(5,45){1}
      \put(45,5){6}
      \put(65,5){7}
      \put(45,25){4}
      \put(65,25){5}
      \put(45,45){2}
    \end{picture}
     \label{grid2}
  }
    \hspace{20pt}
  \subfloat[][]{
    \begin{picture}(80,80)
     \put(0,0){\line(0,1){80}}
      \color{red}\put(39,0){\line(0,1){40}}
      \color{red}\put(41,0){\line(0,1){40}}
      \color{blue} \put(40,40){\line(0,1){40}} %\multiput(40,40)(0,5.25){8}{\line(0,1){3}}
      %\color{blue}\put(60,0){\line(0,1){40}}
      \color{black}\put(80,0){\line(0,1){80}}
      %\color{blue}\put(40,20){\line(1,0){40}}
      \color{black}\put(0,0){\line(1,0){80}}
      \color{red}\put(0,39){\line(1,0){80}}
      \color{red}\put(0,41){\line(1,0){80}}
      \color{black}\put(0,80){\line(1,0){80}}
       \color{black}\put(40,40){\circle*{6}}
       \color{black} \put(5,5){3}
      \put(5,45){1}
      \put(45,5){4}
      \put(45,45){2}
    \end{picture}
     \label{grid3}
     }
         \hspace{20pt}
  \subfloat[][]{
    \begin{picture}(80,80)
     \put(0,0){\line(0,1){80}}
      \color{red}\put(39,0){\line(0,1){40}}
      \color{red}\put(41,0){\line(0,1){40}}
      \color{blue}\put(40,40){\line(0,1){40}}%\multiput(40,40)(0,5.25){8}{\line(0,1){3}} 
      \color{blue} \put(60,0){\line(0,1){39}}%\multiput(60,0)(0,5.25){8}{\line(0,1){3}}
      \color{blue} \put(20,0){\line(0,1){39}}%\multiput(20,0)(0,5.25){8}{\line(0,1){3}}
      \color{black}\put(80,0){\line(0,1){80}}
      \color{blue} \put(41,20){\line(1,0){39}}%\multiput(40,20)(5.25,0){8}{\line(1,0){3}}
      \color{blue} \put(0,20){\line(1,0){39}}%\multiput(0,20)(5.25,0){8}{\line(1,0){3}}
      \color{black}\put(0,0){\line(1,0){80}}
      \color{red}\put(0,39){\line(1,0){80}}
      \color{red}\put(0,41){\line(1,0){80}}
      \color{black}\put(0,80){\line(1,0){80}}
       \color{black}\put(40,40){\circle*{6}}
       \color{black} \put(5,5){5}
      \put(5,45){1}
       \put(25,5){6}
      \put(5,25){3}
      \put(25,25){4}
      \put(45,5){9}
      \put(65,5){10}
      \put(45,25){7}
      \put(65,25){8}
      \put(45,45){2}
    \end{picture}
     \label{grid4}
  }
  \caption{  Special grid structures that have to be studied to maintain symmetry. The red lines denote that the nonconforming block interfaces are treated with SBP interpolation in combination with SAT coupling terms, and the blue dashed line that c-FD discretization is used over the block interface.}
   \label{grids}
\end{figure}
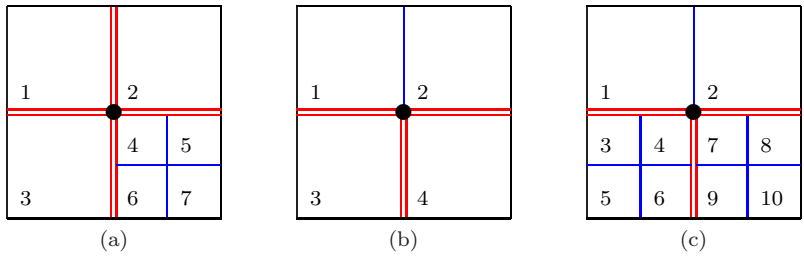

%In \cite{Nissen11} interfaces between nonconforming grid blocks have been discussed. For a fully adaptive grid, more complicated configurations can occur. Special attention has to be paid for corner points such as the black point in Fig.~\ref{grid2}. Interfaces between blocks of different refinement level must be treated using SBP-SAT penalty terms in order to get a stable discretization. 

As mentioned above, interfaces between blocks of identical refinement are treated either using one-sided stencils and penalty terms (SBP-SAT interface), or using c-FD stencils (FD interface). The SBP-SAT framework implies that there are two separate solutions along an interface, one in each block, that are allowed to differ. For FD interfaces on the other hand, only one unique solution along the interface is allowed. In our implementation, in order to swiftly be able to alter the representation of an interface, we always keep separate copies of the solution on either side of the interface. By enforcing the values to be equal for both solutions along every FD interface, this will not affect the numerical solution in any way.

Since the error constant is smaller for c-FD stencils than for SBP-SAT stencils (cf. the experiments in Sec.~\ref{sec:ho}--\ref{sec:mesh_design}), it is reasonable to use FD interfaces whenever possible. The immediate strategy that comes to mind would be to use FD interfaces for all boundaries between equally refined blocks and SBP-SAT interfaces elsewhere. However, a corner point as the black point in Fig.~\ref{grid2} would lead to an asymmetric discretization if treated this way. Adding SBP-SAT interfaces between blocks 1 \& 2 and 1 \& 3 solves this problem but if we further refine block 3, ending up with the situation in Fig.~\ref{grid4}, it is no longer desirable to have an SBP-SAT interface between blocks 1 \& 2. In order to enforce an FD interface along this block boundary in a stable manner, we need to take special care of the grid points around the black point in Fig.~\ref{grid4} (and similarly in Fig.~\ref{grid3}). We refer to this type of intersection as SBP-FD junctions, and in the following subsection we devise a stable treatment of them.

% If we want to have the possibility to coarsen further in the next layers (cf.~Fig.~\ref{fig:example3_mesh_b}), it is desirable to be able to switch between SBP-SAT and FD interfaces, to avoid additional SBP-SAT interfaces imposed due to stability. 
%Thus, we need to devise a stable treatment of the cases depicted in Fig.~\ref{grid3} and Fig.~\ref{grid4}. In the following subsection, we will discuss how to design such an SBP-FD junction in a stable manner.

\subsection{Stable treatment of SBP-FD junctions}\label{sec:junction}

We consider the semi-discretization of Eq.~\eqref{TDSE} for the case of Fig.~\ref{grid3} and denote the semi-discrete solution in block 1 by $u$, in block 3 by $v$, and the joint solution in blocks 2 and 4 by $w$. In order to obtain a stable semi-discretization we use the energy method. We define the scalar product and the norm of vectors $f$, $g$ as
\begin{align}
%\label{scalarproduct}
(f,g)_{P} = f^* P g, \quad \| f\|_{P}^2 = f^* P f,
\end{align}
where $^*$ denotes the complex conjugate. Here, $P$ is a diagonal, positive definite operator. The Kronecker product is denoted by $\otimes$ and for block $\star$ we have $P_{\star} = P_{y,\star} \otimes P_{x,\star}$. Further, we define $e_{0,\star} = (1,0,\cdots,0)^T$ and $e_{N,\star} = (0, \cdots, 0, 1)^T$ of dimension $n_{y,\star} \times 1$, where $n_{y,\star}$ is the number of points in the $y$-direction for block $\star$. We introduce $D2_{x/y,\star}$'s that are approximations of second derivatives defined by
\begin{equation}\begin{aligned}
D2_{x,w} &= - P_{x,w}^{-1} A_{x,w}, \hspace{10pt} D2_{x,u} = - P_{x,u}^{-1} A_{x,u}, \hspace{10pt} D2_{x,v} = - P_{x,v}^{-1} A_{x,v}, \notag \\ 
D2_{y,w} &= P_{y,w}^{-1} \left( -A_{y,w}  - e_{0,w} e_{0,w}^T S_w \right), \notag \\
D2_{y,u} &= P_{y,u}^{-1} \left( -A_{y,u}  + e_{N,u} e_{N,u}^T S_u \right), \notag \\
D2_{y,v} &= P_{y,v}^{-1} \left( -A_{y,v}  + e_{N,v} e_{N,v}^T S_v \right), \notag
\end{aligned}\end{equation}
where $A_{\star} = A_{\star}^T$, and $A_{\star}$ is positive definite. The first (last) row of the matrix corresponding to the operator $S_w$ ($S_u/S_v$) approximates a first derivative. We use the operators approximating second derivatives developed in \cite{Mattsson2004}.  Since we only consider the coupling at the junction, we have excluded the parts of the operators associated with other boundaries and interfaces. Note that we also need to impose penalty terms between block 3 \& 4 due to the SBP-SAT interface. However, we disregard this term in our analysis since it does not interact with the SBP-FD junction. The approximation of the Laplacian is given as the Kronecker product $D2_{\star} = D2_{x,\star} \otimes I_{y,\star} + I_{x,\star} \otimes D2_{y,\star}$.

Following the theory in \cite{Nissen11}, we make the following ansatz for our spatial discretization of Eq. \eqref{TDSE},
%\begin{equation}
\begin{align}
%\label{eq_w_u_v}
w_t =&\; \im \left\{ D2_{x,w} \otimes I_y +  I_x \otimes D2_{y,w}\right\} w \notag\\
 &-\gamma_w I_x \otimes \left(P_{y,w}^{-1}S_{w}^T \right) \left\{\left( I_{w}^{w} \otimes \left( e_{0,w} e_{0,w}^T\right)\right)w-\vphantom{\left(\begin{array}{c} e_{N,u} \\ e_{N,v} \end{array} \right)^T} \right. \notag \\
 & \mspace{180mu} \left. \left(I_{uv}^w \otimes \left(e_{0,w}\left(\begin{array}{c} e_{N,u} \\ e_{N,v} \end{array} \right)^T\right)\right) \left(\begin{array}{c} u\\v\end{array} \right) \right\} \label{eq_w} \\
&- \tau_w I_x \otimes P_{y,w}^{-1} \left\{ \left( I_w^w\otimes \left( e_{0,w} e_{0,w}^T S_w\right) \right) w-\vphantom{\left(\begin{array}{c} e_{N,u} \\ e_{N,v} \end{array} \right)^T} \right. \notag \\
&\mspace{130mu} \; \left. \left( I_{uv}^w \otimes\left( e_{0,w} \left(\begin{array}{c} e_{N,u} \\ e_{N,v} \end{array} \right)^T  \left(\begin{array}{cc} S_{u} & 0 \\ 0 & S_v \end{array}\right) \right)\right) \left(\begin{array}{c} u\\v\end{array} \right) \right\}, \notag \\
u_t =&\; \im \left\{ D2_{x,u} \otimes I_y +  I_x \otimes D2_{y,u}\right\} u \notag\\& - 
\gamma_{uv} I_x\otimes \left(P_{y,u}^{-1} S_{u}^T\right)\left\{\left(I_u^u \otimes \left( e_{N,u} e_{N,u}^T\right)\right) u - \left(I_w^u \otimes \left( e_{N,u} e_{0,w}^T\right)\right) w \right\}   \label{eq_u} \\
&- \tau_{uv}I_x \otimes P_{y,u}^{-1}  \left\{\left(I_u^u \otimes \left(e_{N,u} e_{N,u}^T  S_u\right)\right) u - \left(I _w^u\otimes \left(e_{N,u} e_{0,w}^T  S_w\right)\right) w \right\}, \notag \\
v_t =&\; \im \left\{ D2_{x,v} \otimes I_y +  I_x \otimes D2_{y,v}\right\} v \notag 
\\&- \gamma_{uv}I_x \otimes\left( P_{y,v}^{-1} S_{v}^T\right) \left\{ \left(I_v^v \otimes \left(e_{N,v} e_{N,v}^T \right)\right) v - \left(I_w^v \otimes \left( e_{N,v} e_{0,w}^T \right)\right) w \right\}\label{eq_v}  \\ 
&- \tau_{uv}I_x \otimes P_{y,v}^{-1}  \left\{\left(I_v^v \otimes \left( e_{N,v} e_{N,v}^T  S_v\right)\right) v -\left(I_w^v \otimes \left( e_{N,v} e_{0,w}^T S_w\right)\right) w \right\}, \notag \\[10pt]
&t  \geq 0, w(0) = w^{(0)}, u(0) = u^{(0)}, v(0) = v^{(0)}. \notag 
%%& \| w(t) \|_h < \infty, \| u(t) \|_h < \infty, \| v(t) \|_h < \infty. \notag
%\label{eq_w_u_v}
\end{align} %\end{equation}
By considering an energy estimate, we can derive penalty parameters \\ $\gamma_w,\tau_w,\gamma_{uv},\tau_{uv} \in \I \mathbb{R}$ and design interpolation operators $I_w^w,I_u^u,I_v^v,I_{uv}^w, I_w^u,I_w^v$ that ensure a stable and accurate discretization. In order to obtain a bound we will need the following relation,
\begin{equation}\label{eq:interpolation_relation}\begin{aligned}
I_{uv}^w = \left( \begin{array}{cc} P_{x,w}^{-1}(I_w^u)^T P_{x,u} &  \hspace{5pt}  P_{x,w}^{-1}(I_w^v)^T P_{x,v} \end{array} \right).
\end{aligned}\end{equation}
Using equation \eqref{eq:interpolation_relation} we arrive at the following theorem for the situation in Fig. \ref{grid3}. 
% Here we need to introduce equation 16 before referring to it in the theorem:
%%%%%%%%%%%%%%%%%%%%
\begin{theorem}
\label{th:sbp-fd-jct-tdse-eq-ref}
Consider the SBP-FD junction discretization for the Schr\"{o}dinger equation \eqref{eq_w}-\eqref{eq_v}, with operators $I_w^w,I_u^u,I_v^v$ given by identity matrices and interpolation operators $I_{uv}^w, I_w^u,I_w^v$ that satisfy Eq. \eqref{eq:interpolation_relation}. The SBP-FD junction is stable by the equality
\begin{align}
\| w(t)\|^2_{P_w} + \| u(t)\|^2_{P_u}  + \| v(t)\|^2_{P_v} = \| w^{(0)}\|^2_{P_w} +  \| u^{(0)}\|^2_{P_u} +  \| v^{(0)}\|^2_{P_v},    \label{thm1}
\end{align}
for all $t \geq 0$, if the penalty parameters are chosen as
\begin{equation}
\gamma_w = -\frac{\im}{2}, \quad \gamma_{uv} = \frac{\im}{2}, \quad \tau_w =\frac{\im}{2}, \quad \tau_{uv} = -\frac{\im}{2}. \notag
\end{equation}
\end{theorem}
 
\begin{proof} 
In order to get a stable discretization we use an energy estimate. Multiplying equations \eqref{eq_w}-\eqref{eq_v} with $w^* P_w$, $u^* P_u$ and $v^* P_v$ from the left, respectively  and adding the transposes leads to a symmetric expression of the form
\begin{equation}\begin{aligned}
&\frac{d}{dt}\| w \|_{P_{w}}^2 + \frac{d}{dt}\| u \|_{P_{u}}^2+ \frac{d}{dt}\| v \|_{P_{v}}^2 = \notag \\ 
&\left( \begin{array}{c}
w_0 \\
u_N \\
v_N \\
(S_w w)_0 \\
(S_u u)_N \\
(S_v v)_N
\end{array} \right)^*
\left( \begin{array}{cccccc}
\multicolumn{3}{c}{}& M_1 &\multicolumn{2}{c}{\text{\framebox[1.4cm][c]{$M_4$}}}\\
\multicolumn{3}{c}{0}& \multirow{2}{*}{\begin{picture}(15,23)
\put(0,0){\framebox(15,23)[c]{$M_5$}}
\end{picture}} & M_2 &0\\ %\framebox[\height][c]{\multirow{2}{*}{$A_5$}}& A_2 & \star\\
\multicolumn{3}{c}{}&  & 0 & M_3\\
M_1^*&\multicolumn{2}{c}{\text{\framebox[1.4cm][c]{$M_5^*$}}} & \multicolumn{3}{c}{}\\
\multirow{2}{*}{\begin{picture}(15,23)
\put(0,0){\framebox(15,23)[c]{$M_4^*$}}
\end{picture}}&M_2^*& 0 & \multicolumn{3}{c}{0}\\
&0&M_3^*& \multicolumn{3}{c}{}\\
\end{array} \right)
\left( \begin{array}{c}
w_0 \\
u_N \\
v_N \\
(S_w w)_0 \\
(S_u u)_N \\
(S_v v)_N
\end{array} \right).
\end{aligned}\end{equation}  
Here $w_0$ denotes the values of $w$ at the interface between blocks 2 \& 4 and blocks 1 \& 3 and $u_N$ and $v_N$ the values of $u$ and $v$ at this interface. In order to conserve the energy, we thus have to cancel out the following terms,
\begin{align}
M_1 &=  -\im P_{x,w} -\tau_w P_{x,w} I_w^w - \gamma_w^* I_w^{w T} P_{x,w},  \label{eqA1} \\
M_2 &=  \im P_{x,u} - \tau_{uv} P_{x,u} I_u^u - \gamma_{uv}^* I_u^{uT} P_{x,u}, \label{eqA2} \\
M_3 &= \im P_{x,v} - \tau_{uv} P_{x,v} I_v^v - \gamma_{uv}^* I_v^{vT} P_{x,v}, \label{eqA3}\\
M_4 &= \tau_w P_{x,w} I_{uv}^w + \gamma_{uv}^* \left( \begin{array}{cc} (I_w^u)^T P_{x,u} & (I_w^v)^T P_{x,v} \end{array} \right),\label{eqA4} \\
M_5 &= \gamma_w^* (I_{uv}^w)^TP_{x,w}  + \tau_{uv}^* \left( \begin{array}{c} P_{x,u} I_w^u  \\ P_{x,v} I_w^v  \end{array} \right).\label{eqA5} 
\end{align}
In order for the expressions given in Eq. \eqref{eqA1}-\eqref{eqA3} to be zero, we need the interpolation operators $I_w^w,I_u^u, I_v^v$ to be identity operators. For $M_4$ and $M_5$ in equations Eq.~\eqref{eqA4}-\eqref{eqA5} to be zero, we need
\begin{equation}\begin{aligned}\label{eq:penalty_interpolation}
	(I_{uv}^w)^TP_{x,w} = \left( \begin{array}{c} P_{x,u} I_w^u  \\ P_{x,v} I_w^v  \end{array} \right).
\end{aligned}\end{equation}
Moreover, the penalty parameters need to satisfy the relations 
\begin{equation}\begin{aligned}
	&-\im -\tau_w-\gamma_w^* = 0, \quad \im -\tau_{uv} -\gamma_{uv}^* = 0, \\
	&\tau_w = - \gamma_{uv}^*, \quad \tau_{uv} = \gamma_w^*.
\end{aligned}\end{equation}
These equations do not define the penalty parameters in a unique way. We may choose these parameters in the same way as for a usual SBP-SAT-interface \cite{Nissen11}, namely
\begin{equation}
	\gamma_w = -\frac{\im}{2}, \quad \gamma_{uv} = \frac{\im}{2}, \quad \tau_w =\frac{\im}{2}, \quad \tau_{uv} = -\frac{\im}{2}. \label{penalty_parameters}
\end{equation}
Moreover, note that Eq.~\eqref{eq:penalty_interpolation} can be rewritten as  Eq.~\eqref{eq:interpolation_relation}. Thus, $M_1$ to $M_5$ are zero. Integration in time yields the equality in Eq. \eqref{thm1}.
\newline \qed 
\end{proof}

\noindent
We also want to augment the theory to the case with different refinement levels illustrated in Fig.~\ref{grid4}. The same ansatz as for the case with a uniform grid is used, we only modify the interpolation operators in Eq.~\eqref{eq_w}-\eqref{eq_v}. The new operators are denoted $\tilde I_{uv}^w,\tilde I_{w}^u,\tilde I_{w}^v$ and the operators associated with the norms for the refined $u$ and $v$ are denoted $\tilde P_u, \tilde P_v$. To obtain a bound we will need the following expression,
\begin{equation}\label{eq:interpolation_withfine}
	\tilde I_{uv}^w = \left( \begin{array}{cc} P_{x,w}^{-1}(\tilde I_w^u)^T \tilde P_{x,u} & \hspace{5pt} P_{x,w}^{-1}(\tilde I_w^v)^T \tilde P_{x,v} \end{array} \right).
\end{equation}
which is the nonconforming equivalent to Eq. \eqref{eq:interpolation_relation}. Using the energy technique again we arrive at the following theorem. 

\begin{theorem} 
\label{th:sbp-fd-jct-tdse-diff-ref}
Consider the SBP-FD junction \eqref{eq_w}-\eqref{eq_v}, with operators $I_w^w,I_u^u,I_v^v$ given by identity matrices and interpolation operators $\tilde I_{uv}^w, \tilde I_w^u, \tilde I_w^v$ that satisfy Eq. \eqref{eq:interpolation_withfine}. The SBP-FD junction is stable by the equality
\begin{align}
\| w(t)\|^2_{P_w} + \| u(t)\|^2_{\tilde P_u}  + \| v(t)\|^2_{\tilde P_v} = \| w^{(0)}\|^2_{P_w} +  \| u^{(0)}\|^2_{\tilde P_u} +  \| v^{(0)}\|^2_{\tilde P_v},    \label{thm2}
\end{align}
for all $t \geq 0$, if the penalty parameters are chosen as
\begin{equation}
\gamma_w = -\frac{\im}{2}, \quad \gamma_{uv} = \frac{\im}{2}, \quad \tau_w =\frac{\im}{2}, \quad \tau_{uv} = -\frac{\im}{2}. \notag
\end{equation}
\end{theorem}

\begin{proof}
We follow the analysis of the proof of Theorem~\ref{th:sbp-fd-jct-tdse-eq-ref}. This leads to the same choice of penalty parameters \eqref{penalty_parameters} and the nonconforming equivalent expression of Eq. \eqref{eq:interpolation_relation}, given in Eq. \eqref{eq:interpolation_withfine}.
%\begin{equation}\label{eq:interpolation_withfine}
%	\tilde I_{uv}^w = \left( \begin{array}{cc} P_{x,w}^{-1}(\tilde I_w^u)^T \tilde P_{x,u} & P_{x,w}^{-1}(\tilde I_w^v)^T \tilde P_{x,v} \end{array} \right).
%\end{equation}
It can easily be verified that Eq. \eqref{eq:interpolation_withfine} is satisfied if we choose
\begin{equation}
	\tilde I_{uv}^w = I_{uv}^w \left(\begin{array}{cc}I_{f2c} & 0 \\ 0 & I_{f2c}\end{array}\right), \quad \tilde I_{w}^u =I_{c2f} I_{w}^u, \quad \tilde I_{w}^v =I_{c2f} I_{w}^v,
\end{equation}
where $I_{f2c}$ and $I_{c2f}$ are the interpolation operators from fine to coarse and from coarse to fine, respectively, that have been derived in \cite{Mattsson10} for SBP-SAT-interfaces.
\newline \qed
\end{proof}

\noindent
Hence, interpolation operators $I_{uv}^w, I_w^u,I_w^v$ that satisfy Eq. \eqref{eq:interpolation_relation} are necessary for a stable discretization. Given these stability requirements, the operators should be chosen such that the local accuracy of the stencil is preserved. In each point we want to approximate the value of the function. The only difficult part is the strip close to the interface where the operators associated with SBP-norms, $P_{x,u},P_{x,v}$, differ from the operator associated with the FD-norm, $P_{x,w}$. We therefore achieve full accuracy and stability when choosing both $I_{uv}^w$ and $I_u^w$ or $I_v^w$, respectively, to be rows of the identity matrix when we are at points where the operators for two norms are equal.

Hence, we only have to study a small strip close to the interface. In the second order case, there is only one point where the norms differ, and we do not actually have to interpolate since the $\frac{1}{2}$ in the SBP-norm nicely reflects the fact that we have two copies of the solution on the SBP side while only having one on the FD-side. For higher orders, the operators have the structure
\begin{align}
I_{uv}^{w} = \left(\begin{array}{ccccccc}
1 & & & & & &  \\
& \ddots & & & & & \\
& & 1 & & & & \\
& & & \bar I_{uv}^w & & & \\
& & & & 1 & &  \\
& & & & & \ddots  & \\
& & & & & & 1
\end{array}\right) \in \mathbb{R}^{n_{x,w}\times(n_{x,u}+n_{x,v})},
\end{align}
\begin{align}
\left(\begin{array}{c c} 
I_w^{u} \\
I_w^{v} \\
\end{array}\right) = \left(\begin{array}{ccccccc}
1 & & & & & &  \\
& \ddots & & & & & \\
& & 1 & & & & \\
& & & \left(\begin{array}{c c} 
\bar I_w^{u} \\
\bar I_w^{v} \\
\end{array}\right) & & & \\
& & & & 1 & &  \\
& & & & & \ddots  & \\
& & & & & & 1
\end{array}\right) \in \mathbb{R}^{(n_{x,u}+n_{x,v}) \times n_{x,w}}.
\end{align}
Here, $\bar I_{uv}^w$ is a $7\times 8$ matrix for the fourth order case and an $11\times 12$ matrix for the sixth order case that preserve order two or three, respectively. Similarly, $\left(\begin{array}{c c} 
\bar I_w^{u} \\
\bar I_w^{v} \\
\end{array}\right)$ is an $8\times 7$ matrix for the fourth order case and a $12\times 11$ matrix for the sixth order case. The order at the junction needs to be reduced compared to the inner stencil, however the accuracy as for the SBP-approximation close to interfaces can be maintained. The matrix $\left(\begin{array}{c c} 
\bar I_w^{u} \\
\bar I_w^{v} \\
\end{array}\right)$ is given in Appendix \ref{appendix_operators} for SBP operators of fourth and sixth order accuracy. Note that the corresponding matrix $\bar I_{uv}^w$ can be constructed from Eq.~\eqref{eq:interpolation_relation}. Further take notice that our interpolation operators couple blocks diagonally across the SBP-FD junction. As an example, consider Fig.~\ref{grid3} in which block 1 does not only interact with block 2 but also with block 4. It would be preferable from a performance point-of-view (to minimize block dependencies and communication) to couple block-wise along the interface. We have tried to accomplish this  but we could then only achieve a first-order accurate coupling at the junction, both for the fourth and the sixth order operators. 

\subsection{Global accuracy}
\label{sec:Accuracy}

For SBP operators, we have to distinguish between the order of accuracy of the inner stencil and the order of accuracy of the stencil close to the boundary. Moreover, we have to consider the accuracy of the SAT penalty terms. Let $2p$ be the order of the inner stencil. Then there is a boundary layer of $p$ points where the accuracy of the derivative approximation is only $p$. Let us now analyze the accuracy of the SAT terms for the block boundaries. The structure of the penalty terms given for SBP-FD junctions in \eqref{eq_w}-\eqref{eq_v} is the same as for boundaries between blocks of different refinement levels. The interpolation operators have in both cases order $2p$ for most points and order $p$ in the boundary layer.  The analysis is done for equation \eqref{eq_u} but the reasoning applies for \eqref{eq_w} and \eqref{eq_v} as well. Let us consider the first SAT term
\begin{equation}
   I_x\otimes \left(P_{y,u}^{-1} S_{u}^T\right)\left\{\left(I_u^u \otimes \left( e_{N,u} e_{N,u}^T\right)\right) u - \left(I_w^u \otimes \left( e_{N,u} e_{0,w}^T\right)\right) w \right\}.
\end{equation}
The term in the curly bracket includes the interpolation operator. Hence, the accuracy is of order $2p$ along the edge, except close to the corner where it is reduced to $p$. Since both $P_{y,u}^{-1}$ and  $S_{u}^T$ are of order $\frac{1}{h}$, the convergence order is reduced to $p-2$ close to the junction point.

Now, turn to the second penalty term,
\begin{equation}
  I_x \otimes P_{y,u}^{-1}  \left\{\left(I_u^u \otimes \left(e_{N,u} e_{N,u}^T  S_u\right)\right) u - \left(I _w^u\otimes \left(e_{N,u} e_{0,w}^T  S_w\right)\right) w \right\}.
\end{equation}
In this term, the derivative is computed first with an accuracy of $p+1$ and thereafter the interpolation is applied. This gives an approximation order for the expression in the curly brackets of $p$ close to the junction and $p+1$ along the edge. Finally, $P_{y,u}^{-1}$ is applied, reducing the order to $p-1$ or $p$, respectively. As a result of the interpolation, we get an order reduction of the scheme to $p-2$ at SBP-FD junction points or corners with different refinement levels. 

For the solution of the time-dependent problem (given a sufficiently accurate integration in time), we can expect the order of accuracy to be at least $p-2$, the lowest order present in the complete stencil. However, the accuracy of the numerical solution is often one or two orders higher than the order of the boundary approximation would suggest \cite{Gustafsson75, Gustafsson81, Svard06}. It has been shown for the Schr\"{o}dinger equation that two orders are gained for the treatment of outer boundaries in one dimension in \cite{Nissen12} and for interface treatment in one dimension in \cite{Nissen11}. Numerical simulations in two dimensions show the same behavior as the one-dimensional analysis  \cite{Nissen11,Nissen12}. To fully understand the two-dimensional behavior close to corner points and SBP-FD points further analysis needs to be carried out. In numerical simulations we observe overall higher convergence orders than expected, based on the reasoning in this section in combination with the theory for gaining orders of accuracy in the time propagation \cite{Gustafsson75,Gustafsson81,Nissen11,Nissen12}. Note that the number of grid points where the order of accuracy is lowered to $p-1$ or $p-2$ is independent of the grid size. This fact could explain the higher than expected order of accuracy in the $\ell_2$ norm. The problem with decreased order at corner points has also been observed by Kramer and co-workers \cite{Kramer09}. For their formulation with continuous stencils for first derivatives, they have observed the same maximum accuracy orders for the stencils at corner points as we have for this case. To fully understand the convergence behavior analysis of the full two-dimensional problem needs to be carried out.

To remedy the loss of accuracy we considered constructing interpolation operators of order $p+1$ on a wider boundary layer according to the algorithm outlined in \cite{Mattsson10}, but we conclude that no such operators exists as the resulting system of linear equations does not have a solution.
%To fully understand the two-dimensional behavior further analysis needs to be carried out. However, in the numerical simulations the convergence behavior lies within the range of what can be expected.

\begin{remark}
In~\cite{Nissen12} the difference between penalties for the Schr\"{o}dinger equation and the diffusion equation was discussed. It was pointed out that $S^T$ penalty terms are a necessity for the Schr\"{o}dinger equation due to its non-diffusive character, whereas SAT terms for the diffusion equation can be formulated without an $S^T$ term. In applications where the second derivative operator is associated with diffusion, such that the $S^T$ terms can be omitted, the lowest order of accuracy will thus be $p-1$ instead of $p-2$.  
\end{remark}

\subsection{Temporal discretization}

After discretization of Eq.~\eqref{TDSE} in space, we are left with the system of ordinary differential equations
\begin{equation}\label{eq:ode}
 \frac{\D}{\D t} \U = - \frac{\I}{\hbar} H \U,
\end{equation}
where $\U$ is the semi-discrete solution and $H$ is the approximated Hamiltonian. If the Hamiltonian is independent of time, the solution of \eqref{eq:ode} can be expressed as
\begin{equation}
	\U (t) = \exp\left(-\frac{\I}{\hbar} H t\right) \U(0). 
\end{equation}
In case $H$ is time-dependent, one can use the exponential form successively on small time intervals. Instead of using just $H$, one has to take a Magnus series expansion to get the exact solution. For numerical purposes, it suffices to take a truncated expansion (cf. \cite{Kormann11}).

Computing the exponential of the discrete Hamiltonian matrix is a computationally intensive task and direct methods are out of reach for realistic grid sizes. Since the matrix $H$ is sparse, Krylov methods provide an efficient alternative. In case the matrix is symmetric, one can use the Lanczos method. Otherwise, one has to take the Arnoldi method which is computationally more intense and has worse scalability properties.
For an SBP discretization, $H$ itself is not symmetric. However, it is symmetric in the norm associated with the SBP operator. We therefore use the Lanczos algorithm and base all norm computations on the SBP norm.

\section{First derivatives}\label{sec:first}

In this section, we consider the advection equation
\begin{equation}
  \label{eq:advection_equation}
  U_t = a_1 U_x + a_2 U_y,
\end{equation}
with initial and boundary conditions, as an example of an equation with first derivatives. Again we use the method of lines approach with finite difference methods in space. We use a fourth order accurate Runge-Kutta method in time. The spatial discretization with SBP-SAT finite differences is done in a similar way as for the Schr{\"o}dinger equation. The SBP stencils for the second derivatives are replaced by the corresponding stencils for the first derivatives. Here, the SBP-SAT boundary treatment becomes simpler since we only need one type of penalty terms that enforces continuity of the solution. 

%At SBP-FD junction points, see Fig.~\ref{grid3}, the semi-discretization reads
The semi-discretization of Eq.~(\ref{eq:advection_equation}) for the SBP-FD junction mesh in Fig.~\ref{grid3} reads as follows
\begin{align}
%\label{eq_w_u_v}
w_t =&  \left\{ a_1 D_{x,w} \otimes I_y +  a_2 I_x \otimes D_{y,w}\right\} w \notag\\
&- \tau_w  I_x \otimes P_{y,w}^{-1}  \left\{ \left( I_w^w \otimes e_{0,w} e_{0,w}^T \right) w- %{\left(\begin{array}{c} e_{N,u} \\ e_{N,v} \end{array} \right)^T}  %\right. \notag \\
 \left( I_{uv}^w \otimes e_{0,w} \left(\begin{array}{c} e_{N,u} \\ e_{N,v} \end{array} \right)^T \right) \left(\begin{array}{c} u\\v\end{array} \right) \right\}, \label{eq_w_adv} \\
u_t =&  \left\{ a_1 D_{x,u} \otimes I_y +  a_2 I_x \otimes D_{y,u}\right\} u \notag\\
&- \tau_{uv} I_x \otimes P_{y,u}^{-1}  \left\{(I_u^u \otimes e_{N,u} e_{N,u}^T ) u - (I _w^u\otimes e_{N,u} e_{0,w}^T ) w \right\},   \label{eq_u_adv} \\
v_t =&  \left\{a_1 D_{x,v} \otimes I_y + a_2 I_x \otimes D_{y,v}\right\} v \notag \\
&- \tau_{uv} I_x \otimes P_{y,v}^{-1} \left\{ (I_v^v \otimes e_{N,v} e_{N,v}^T ) v -(I_w^v \otimes e_{N,v} e_{0,w}^T) w \right\}, \label{eq_v_adv}  \\[10pt]
&t  \geq 0, w(0) = w^{(0)}, u(0) = u^{(0)}, v(0) = v^{(0)}. \notag 
%& \| w(t) \|_h < \infty, \| u(t) \|_h < \infty, \| v(t) \|_h < \infty. \notag
%\label{eq_w_u_v}
%&\qquad \quad  \left.
\end{align}
The $D_{\star}$'s are approximations of the first derivative, satisfying the SBP property $D_{x,\star} = P_{x,\star}^{-1}Q_{x,\star}$, $D_{y,\star} = P_{y,\star}^{-1}Q_{y,\star}$. We have that $Q_{y,w}+Q_{y,w}^T = \mbox{diag}[-1,0, \cdots, 0]$, and $Q_{y,u}+Q_{y,u}^T = Q_{y,v}+Q_{y,v}^T = \mbox{diag}[0,0, \cdots, 0, 1]$. Note that the discretization in Eq.~\eqref{eq_w_adv}-\eqref{eq_v_adv} does not include an $S^T$ penalty. As a consequence only one order of accuracy is lost at junction- and corner points, i.e., the total accuracy of the derivative approximation is of order $p-1$. Since this is a discretization of a first derivative, from one-dimensional analysis we only expect to gain one order in the numerical solution with respect to the stencil order  \cite{Gustafsson75,Gustafsson81}. In the two-dimensional simulations in this paper we see convergence orders that are better, especially in the $\ell_2$ norm. The reason for this could be that the number of grid points with lower accuracy is fixed and thus play a less important role as the grid is refined. Two-dimensional analysis is necessary to get the full picture.

A stability estimate for Eq.~\eqref{eq_w_adv}-\eqref{eq_v_adv} can be derived using the energy method. As in the Schr\"{o}dinger-case we have omitted penalty terms for exterior boundaries as well as the interface coupling between $u$ and $v$. The stability analysis for the SBP-FD junction with a uniform mesh (Fig.~\ref{grid3}) yields the following theorem.

\begin{theorem}
\label{th:sbp-fd-jct-adv-eq-eq-ref}
Consider the SBP-FD junction discretization for the advection equation \eqref{eq_w_adv}-\eqref{eq_v_adv}, with operators $I_w^w,I_u^u,I_v^v$ given by identity matrices and interpolation operators $I_{uv}^w, I_w^u,I_w^v$ that satisfy Eq. \eqref{eq:interpolation_relation}. The SBP-FD junction is stable by the equality
\begin{align}
\| w(t)\|^2_{P_w} + \| u(t)\|^2_{P_u}  + \| v(t)\|^2_{P_v} = \| w^{(0)}\|^2_{P_w} +  \| u^{(0)}\|^2_{P_u} +  \| v^{(0)}\|^2_{P_v},    \label{thm3}
\end{align}
for all $t \geq 0$, if the penalty parameters are chosen as
\begin{equation}
\tau_w = -\frac{a_2}{2}, \quad \tau_{uv} = \frac{a_2}{2}. \notag
\end{equation}
\end{theorem}
 
\begin{proof} 
We follow the proof for the stability analysis for the Schr\"{o}dinger equation with uniform grid. Multiplying equations \eqref{eq_w_adv}-\eqref{eq_v_adv} with $w^* P_w$, $u^* P_u$ and $v^* P_v$ from the left and adding the transposes leads to the expression 
\begin{equation}\begin{aligned}
&\frac{d}{dt}\| w \|_{P_{w}}^2 + \frac{d}{dt}\| u \|_{P_{u}}^2+ \frac{d}{dt}\| v \|_{P_{v}}^2 = \notag \\ 
&\left( \begin{array}{c}
w_0 \\
u_N \\
v_N 
\end{array} \right)^T
\left( \begin{array}{cccccc}
%\multicolumn{3}{c}{}& M_1 &\multicolumn{2}{c}{\text{\framebox[1.4cm][c]{$M_4$}}}\\
%\multicolumn{3}{c}{0}& \multirow{2}{*}{\begin{picture}(15,23)
%\put(0,0){\framebox(15,23)[c]{$M_5$}}
%\end{picture}} & M_2 &0\\ %\framebox[\height][c]{\multirow{2}{*}{$A_5$}}& A_2 & \star\\
%\multicolumn{3}{c}{}&  & 0 & M_3\\
M_1&\multicolumn{2}{c}{\text{\framebox[1.4cm][c]{$M_5$}}} \\ %& \multicolumn{3}{c}{}\\
\multirow{2}{*}{\begin{picture}(15,23)
\put(0,0){\framebox(15,23)[c]{$M_4$}}
\end{picture}}&M_2& 0 \\ %& \multicolumn{3}{c}{0}\\
&0&M_3 \\ %& \multicolumn{3}{c}{}\\
\end{array} \right)
\left( \begin{array}{c}
w_0 \\
u_N \\
v_N
\end{array} \right),
\end{aligned}\end{equation}  
where
\begin{align}
M_1 &= (-a_2 - 2\tau_w) P_{x,w}, \\
M_2 &= (a_2 - 2\tau_{uv}) P_{x,u}, \\
M_3 &= (a_2 - 2\tau_{uv}) P_{x,v}, \\
M_4 &= \tau_w (I_{uv}^w)^TP_{x,w}  +  \tau_{uv} \left( \begin{array}{c} P_{x,u} I_w^u  \\  P_{x,v} I_w^v  \end{array} \right), \\ 
M_5 &= \tau_w P_{x,w} I_{uv}^{w} + \tau_{uv} \left( \left(I_w^{u}\right)^T P_{x,u} \hspace{5pt}  \left(I_w^{v}\right)^T P_{x,v} \right), \label{eqA5_a}
\end{align}
assuming that $I_w^w$, $I_u^u$, $I_v^v$ are identity matrices. By using relation \eqref{eq:penalty_interpolation} $M_4$ and $M_5$ are zero if $\tau_{uv} = -\tau_w$. With $\tau_w = -\frac{a_2}{2}$, $\tau_{uv} = \frac{a_2}{2}$, $M_1$, $M_2$ and $M_3$ are zero. By integrating in time we arrive at Eq. \eqref{thm3}.
\newline \qed
\end{proof}

\noindent
For the stability analysis of the SBP-FD junction with different levels of refinement (Fig.~\ref{grid4}), we arrive at the following theorem.
\begin{theorem}
\label{th:sbp-fd-jct-adv-eq-diff-ref}
Consider the SBP-FD junction discretization for the advection equation \eqref{eq_w_adv}-\eqref{eq_v_adv}, with operators $I_w^w,I_u^u,I_v^v$ given by identity matrices and interpolation operators $\tilde I_{uv}^w, \tilde I_w^u, \tilde I_w^v$ that satisfy Eq. \eqref{eq:interpolation_withfine}. The SBP-FD junction is stable by the equality
\begin{align}
\| w(t)\|^2_{P_w} + \| u(t)\|^2_{\tilde P_u}  + \| v(t)\|^2_{\tilde P_v} = \| w^{(0)}\|^2_{P_w} +  \| u^{(0)}\|^2_{\tilde P_u} +  \| v^{(0)}\|^2_{\tilde P_v},    \label{thm4}
\end{align}
for all $t \geq 0$, if the penalty parameters are chosen as
\begin{equation}
\tau_w = -\frac{a_2}{2}, \quad \tau_{uv} = \frac{a_2}{2}. \notag
\end{equation}
\end{theorem}
 
\begin{proof}
%The proof leading to Eq. \eqref{thm4} follows the proofs for Eq. \eqref{thm3} and Eq. \eqref{thm2}. With the same choice of penalty parameters as for the situation in Fig.~\ref{grids}b and using relation \eqref{eq:interpolation_withfine} we arrive at Eq. \eqref{thm4} after integration in time.
The proof leading to Eq. \eqref{thm4} follows the proofs for Theorem~\ref{th:sbp-fd-jct-adv-eq-eq-ref} and Theorem~\ref{th:sbp-fd-jct-tdse-diff-ref}. With the same choice of penalty parameters as for the situation in Fig.~\ref{grids}b and using relation \eqref{eq:interpolation_withfine} we arrive at Eq. \eqref{thm4} after integration in time.
\newline \qed
\end{proof}

\noindent
We remark that the penalty parameters close to the SBP-FD junction can be used along the rest of the interface as well, as was the case for the Schr\"{o}dinger equation. This makes the extension to include SBP-FD junctions in an already existing code with SBP-SAT interfaces straightforward. 

\section{Numerical convergence study}%\label{sec:numerical_convergence}
\label{sec:num_convergence}

In this section, we provide numerical convergence studies for the free Schr{\"o}dinger equation and the advection equation. We study the three important cases given by the meshes in Fig. \ref{grids}. We use stencils with inner order two, four, and six. Each experiment starts with a coarse base grid where each block has $21 \times 21$ elements (refinement level 0). By successive isotropic refinement of the individual blocks in each grid we repeat the experiments up to refinement level 5 ($641 \times 641$ elements). Since we are focusing on the spatial discretization, we use time steps small enough for the temporal error to be negligible compared to the spatial error.

\subsection{Free Schr{\"o}dinger equation}
\label{sec:num_convergence:TDSE}
%[TODO: The discussion in this subsection needs to be rewritten.]
As initial value to Eq. \eqref{TDSE} with $V=0$, we use a Gaussian, 
 \begin{equation}\begin{aligned}
 	U(\x,0) =& \exp\left(-\alpha_x (x-x_0)^2+\I k_x (x-x_0) \right) \cdot \\ 
	&\exp\left(- \alpha_y (y-y_0)^2+\I k_y (y-y_0)\right). \label{Gaussian}
\end{aligned} \end{equation}
The equation is closed by periodic boundary conditions. 
The parameters are chosen as $\alpha_x = \alpha_y = 1$, $x_0 = y_0 = k_x = k_y = 0$ and the mesh covers the domain $[-10,10]\times[-10,10]$. This means that the wave packet is centered at the corner points marked in black in each respective grid in Fig.~\ref{grids}. All three meshes are badly suited for this kind of wave packet since there are difficult grid boundaries and corners right at the top of the function. These examples should be viewed as worst case scenarios to demonstrate the capabilities of our framework, not as good examples of a mesh fitted to such a function.
%All three meshes are very badly suited for this kind of wave packet since the structure of the grid changes exactly at the top of the function. However, these examples are suited to show that the solution converges also at interfaces and corners where the structure and the refinement level change and to analyze our newly constructed SBP-FD-junction operators numerically. 

The simulation time is $t=0.05$. We present the errors and convergence rates compared to an analytical solution \cite{Tannor07} in Tables \ref{tab:schr_example_grid2}--\ref{tab:schr_example_grid4}. The expected order of accuracy for an interior stencil of order $2p$ is $p$ along edges, and $p-2$ at corner points for a second order derivative approximation. A general trend in the numerical experiments is that in $\ell_{\infty}$ norm we see a gain of two orders in accuracy compared to the accuracy at corner points. In the $\ell_2$ norm the accuracy order is often an additional order higher. This is probably due to the fact that the accuracy order is lower only in a limited number of points. The convergence rates are not precise for all cases, which can be explained by the fact that various error terms that converge at different rates are present. Similar behavior was seen in \cite{Nissen11, Nissen12}, where the lowest error terms were dominating only for very fine grids for spatial discretizations of order 6 and 8.

\begin{table}[H] \renewcommand{\arraystretch}{1.3}
\caption{Convergence for grid Fig.~\ref{grid2} for the Schr\"{o}dinger equation (Sec.~\ref{sec:num_convergence:TDSE}). The number of isotropic refinements done in each block is given in the first column.}
\label{tab:schr_example_grid2}
\centering \begin{tabular}{c||c||c|c||c|c} \hline
order&\bfseries  & \bfseries $\ell_2$ error & \bfseries conv. rate &\bfseries $\ell_{\infty}$ error & \bfseries conv. rate \\ 
\hline\hline 
2 &0 & $4.8 \cdot 10^{-2}$ & --- & $1.0 \cdot 10^{-1}$ & --- \\
&1 & $9.5 \cdot 10^{-3}$ & 2.3 & $1.9 \cdot 10^{-2}$ & 2.4 \\
&2 & $1.5 \cdot 10^{-3}$ & 2.7 & $4.1 \cdot 10^{-3}$ & 2.2 \\
&3 & $ 3.1 \cdot 10^{-4}$ & 2.3 & $8.4 \cdot 10^{-4}$ & 2.3 \\
&4 & $7.1 \cdot 10^{-5}$ & 2.1 & $3.1\cdot 10^{-4}$ & 1.4 \\
&5 & $1.7 \cdot 10^{-5}$ & 2.1 & $5.9 \cdot 10^{-5}$ & 2.4 \\
\hline 
4 &0 & $5.9 \cdot 10^{-2}$ & --- & $9.1 \cdot 10^{-2}$ & --- \\
&1 & $9.2  \cdot 10^{-3}$ & 2.7 & $1.8 \cdot 10^{-2}$ & 2.3 \\
&2 & $1.6  \cdot 10^{-3}$ & 2.6 & $6.8 \cdot 10^{-3}$ & 1.4 \\
&3 & $5.6 \cdot 10^{-4}$ & 1.5 & $4.8 \cdot 10^{-3}$ & 0.5 \\
&4 & $1.5 \cdot 10^{-4}$ & 1.9 & $2.6 \cdot 10^{-3}$ & 0.9 \\
&5 & $6.9 \cdot 10^{-6}$ & 4.4 & $2.4 \cdot 10^{-4}$ & 3.5 \\
\hline
6 &0 & $1.3 \cdot 10^{-1}$ & --- & $1.5 \cdot 10^{-1}$ & --- \\
&1 & $1.3 \cdot 10^{-2}$ & 3.3 & $1.7 \cdot 10^{-2}$ & 3.2 \\
&2 & $3.5 \cdot 10^{-4}$ & 5.2 & $8.9 \cdot 10^{-4}$ & 4.2 \\
&3 & $1.6 \cdot 10^{-5}$ & 4.4 & $1.1 \cdot 10^{-4}$ & 3.0 \\
&4 & $7.6 \cdot 10^{-7}$ & 4.4 & $5.4 \cdot 10^{-6}$ & 4.3 \\
&5 & $3.9 \cdot 10^{-8}$ & 4.3 & $3.6 \cdot 10^{-7}$ & 3.9 \\
\end{tabular} \end{table}

% With stencil-based code
%\begin{table}[!t] \renewcommand{\arraystretch}{1.3}
%\caption{Convergence for grid Fig.~\ref{grid3}. The number of isotropic refinements done in each block is given in the first column.}
%\label{tab:example0_grid2}
%\centering \begin{tabular}{c||c||c|c||c|c} \hline
%order&\bfseries  & \bfseries $\ell_2$ error & \bfseries conv. rate &\bfseries $\ell_{\infty}$ error & \bfseries conv. rate \\ 
%\hline\hline 
%2 &0 & $8.4\cdot 10^{-2}$ & --- & $7.6\cdot 10^{-2}$ & --- \\
%&1 & $1.7\cdot 10^{-2}$ & 2.3 & $2.0\cdot 10^{-2}$ & 1.9 \\
%&2 & $4.0\cdot 10^{-3}$ & 2.1 & $5.0\cdot 10^{-3}$ & 3.5 \\
%&3 & $9.3\cdot 10^{-4}$ & 2.1 & $1.3\cdot 10^{-3}$ & 1.9\\
%&4 & $2.3\cdot 10^{-4}$ & 2.0 & $3.1\cdot 10^{-4}$ & 2.1\\
%\hline 
%4 &0 & $4.9\cdot 10^{-2}$ & --- & $5.5\cdot 10^{-2}$ & --- \\
%&1 & $7.3\cdot 10^{-3}$ & 2.7 & $1.4\cdot 10^{-2}$ & 2.0 \\
%&2 & $6.6\cdot 10^{-4}$ & 3.5 & $1.6\cdot 10^{-3}$ & 3.1 \\
%&3 & $1.6\cdot 10^{-4}$ & 2.0 & $8.4\cdot 10^{-4}$ & 0.9\\
%&4 & $2.0\cdot 10^{-5}$ & 3.0 & $2.0\cdot 10^{-4}$ & 2.1\\
%\hline
%6 &0 & $1.0\cdot 10^{-1}$ & --- & $9.1\cdot 10^{-2}$ & --- \\
%&1 & $9.4\cdot 10^{-3}$ & 3.4 & $1.2\cdot 10^{-2}$ & 2.9 \\
%&2 & $2.9\cdot 10^{-4}$ & 5.0 & $7.9\cdot 10^{-4}$ & 3.9 \\
%&3 & $5.7\cdot 10^{-5}$ & 2.3 & $2.9\cdot 10^{-4}$ & 1.4\\
%&4 & $2.7\cdot 10^{-6}$ & 4.4 & $2.8\cdot 10^{-5}$ & 3.4\\
%\end{tabular} \end{table}

\begin{table}[H] \renewcommand{\arraystretch}{1.3}
\caption{Convergence for grid Fig.~\ref{grid3} for the Schr\"{o}dinger equation (Sec.~\ref{sec:num_convergence:TDSE}). The number of isotropic refinements done in each block is given in the first column.}
\label{tab:schr_example_grid3}
\centering \begin{tabular}{c||c||c|c||c|c} \hline
order&\bfseries  & \bfseries $\ell_2$ error & \bfseries conv. rate &\bfseries $\ell_{\infty}$ error & \bfseries conv. rate \\ 
\hline\hline 
2 &0 & $2.9\cdot 10^{-2}$ & --- & $2.2\cdot 10^{-2}$ & --- \\
   &1 & $6.1\cdot 10^{-3}$ & 2.2 & $5.8\cdot 10^{-3}$ & 1.9 \\
   &2 & $1.4\cdot 10^{-3}$ & 2.1 & $1.5\cdot 10^{-3}$ & 2.0 \\
   &3 & $3.3\cdot 10^{-4}$ & 2.1 & $3.7\cdot 10^{-3}$ & 2.0\\
   &4 & $7.9\cdot 10^{-5}$ & 2.0 & $9.2\cdot 10^{-5}$ & 2.0\\
   &5 & $2.0 \cdot 10^{-5}$ & 2.0 & $2.3 \cdot 10^{-5}$ & 2.0 \\
\hline 
4 &0 & $4.9\cdot 10^{-2}$ & --- & $5.5\cdot 10^{-2}$ & --- \\
   &1 & $7.3\cdot 10^{-3}$ & 2.8 & $1.4\cdot 10^{-2}$ & 2.0 \\
   &2 & $6.6\cdot 10^{-4}$ & 3.5 & $1.6\cdot 10^{-3}$ & 3.1 \\
   &3 & $8.4\cdot 10^{-5}$ & 3.0 & $4.6\cdot 10^{-4}$ & 1.8\\ 
   &4 & $1.0\cdot 10^{-5}$ & 3.0 & $1.1\cdot 10^{-4}$ & 2.0\\
   &5 & $1.3 \cdot 10^{-6}$ & 3.0 & $2.7 \cdot 10^{-5}$ & 2.0 \\
\hline
6 &0 & $1.1\cdot 10^{-1}$ & --- & $9.1\cdot 10^{-2}$ & --- \\
   &1 & $9.4\cdot 10^{-3}$ & 3.5 & $1.2\cdot 10^{-2}$ & 2.9 \\
   &2 & $2.9\cdot 10^{-4}$ & 5.0 & $7.9\cdot 10^{-4}$ & 3.9 \\
   &3 & $1.1\cdot 10^{-5}$ & 4.7 & $5.4\cdot 10^{-5}$ & 3.9\\
   &4 & $5.7\cdot 10^{-7}$ & 4.3 & $5.8\cdot 10^{-6}$ & 3.2\\
   &5 & $2.8 \cdot 10^{-8}$ & 4.4 & $5.6 \cdot 10^{-7}$ & 3.4 \\
\end{tabular} \end{table}

\begin{table}[H] \renewcommand{\arraystretch}{1.3}
\caption{Convergence for grid Fig.~\ref{grid4} for the Schr\"{o}dinger equation (Sec.~\ref{sec:num_convergence:TDSE}). The number of isotropic refinements done in each block is given in the first column.}
\label{tab:schr_example_grid4}
\centering \begin{tabular}{c||c||c|c||c|c} \hline
order&\bfseries  & \bfseries $\ell_2$ error & \bfseries conv. rate &\bfseries $\ell_{\infty}$ error & \bfseries conv. rate \\ 
\hline\hline 
2 &0 & $4.3 \cdot 10^{-2}$ & --- & $6.4 \cdot 10^{-2}$ & --- \\
   &1 & $6.4 \cdot 10^{-3}$ & 2.7 & $1.5 \cdot 10^{-2}$ & 2.1 \\
   &2 & $1.1 \cdot 10^{-3}$ & 2.6 & $1.8 \cdot 10^{-3}$ & 3.1 \\
   &3 & $2.5 \cdot 10^{-4}$ & 2.1 & $3.5 \cdot 10^{-4}$ & 2.3\\
   &4 & $5.8 \cdot 10^{-5}$ & 2.1 & $1.1 \cdot 10^{-4}$ & 1.7 \\
   &5 & $1.4 \cdot 10^{-5}$ & 2.0 & $3.1 \cdot 10^{-5}$ & 1.8 \\
\hline 
4 &0 & $5.8 \cdot 10^{-2}$ & --- & $1.2 \cdot 10^{-1}$ & --- \\
   &1 & $5.3 \cdot 10^{-3}$ & 3.5 & $1.3 \cdot 10^{-2}$ & 3.2 \\
   &2 & $5.3 \cdot 10^{-4}$ & 3.3 & $1.6 \cdot 10^{-3}$ & 3.0 \\
   &3 & $8.8 \cdot 10^{-5}$ & 2.6 & $4.8 \cdot 10^{-4}$ & 1.7\\
   &4 & $1.0 \cdot 10^{-5}$ & 3.1 & $1.2 \cdot 10^{-4}$ & 2.0\\
   &5 & $1.3 \cdot 10^{-6}$ & 3.0 & $3.3 \cdot 10^{-5}$ & 1.9 \\
\hline
6 &0 & $1.3 \cdot 10^{-1}$ & --- & $2.8 \cdot 10^{-1}$ & --- \\
   &1 & $1.9 \cdot 10^{-2}$ & 2.8 & $3.9 \cdot 10^{-2}$ & 2.9 \\
   &2 & $1.7 \cdot 10^{-3}$ & 3.5 & $8.4 \cdot 10^{-3}$ & 2.2 \\
   &3 & $1.4 \cdot 10^{-4}$ & 3.6 & $6.9 \cdot 10^{-4}$ & 3.6\\
   &4 & $9.5 \cdot 10^{-6}$ & 3.9 & $4.5 \cdot 10^{-5}$ & 3.9\\
   &5 & $7.6 \cdot 10^{-7}$ & 3.6 & $5.8 \cdot 10^{-6}$ & 3.0
\end{tabular} \end{table}

\subsection{Advection equation}
\label{subs:advection_equation}

With $a_1 = -2$, $a_2 = -\frac{6}{23}$ in Eq.~\eqref{eq:advection_equation}, an analytic solution to the advection equation is given by
\[
	U(x,y,t) = e^{-50 \left(\left(x- \frac{4}{5} - 2 t\right)^2 + \left(y - \frac{4}{5} - \frac{6}{23} t\right)^2\right)}.
\]
In this example, we let the grids cover the domain $[0,4] \times [0,4]$ and use periodic boundary conditions. The simulation time is $t = 4.6$; at this time the solution is centered exactly at the previously described critical corners. The experiments are repeated six times with isotropic refinement in each repetition. 

%Errors and convergence rates for the experiments are listed in Tables \ref{tab:adv_example_grid2}--\ref{tab:adv_example_grid4}. 
%In this study we expect at most first, second and third order convergence rates in the $\ell_2$ norm, corresponding to SBP operators of interior order 2, 4 and 6. 

Tables \ref{tab:adv_example_grid2}--\ref{tab:adv_example_grid4} list the errors and rates of convergence of the numerical solution in $\ell_2$ norm and $\ell_\infty$ norm. For first derivatives the expected convergence rate along edges is $p$ for a $2p$ interior stencil, and $p-1$ at corner points as presented in sec. \ref{sec:Accuracy}. The numerical simulations show that we obtain convergence rates that are approximately one order higher than the expected rates along edges, or two orders higher than expected at corner points. Therefore the lower accuracy close to corner points does not seem to affect the overall accuracy as strongly for the advection equation as for the Schr\"{o}dinger equation. This is likely due to that here we only have one penalty term, and it does not contain an $S^T$ term, thus affecting fewer grid points around the interface. Another reason could be that we have not entered the convergence region where the errors due to the grid points with lower order accuracy are dominating. Compared to the experiment with the Schr{\"o}dinger equation the corner point is also less important since the solution is only centred at the corner point at the final time. Similarly as for the Schr\"{o}dinger equation, the $\ell_2$ convergence rates are somewhat higher than the ones in $\ell_{\infty}$ norm, likely due to the finite number of grid points with lower accuracy order. 

%The convergence rates we obtain are for most simulations better than expected, especially in the $\ell_2$ norm. 

%We obtain the expected optimal convergence rates in the $\ell_\infty$ norm, whereas in the $\ell_2$ norm the convergence rates are in some cases a bit better than expected. 
%Since the error with the lower convergence rate is localized around the interfaces, we expect that we may need additional refinement to the mesh for this error to come through in the $\ell_2$ norm, similarly to the results in Sec. \ref{sec:num_convergence:TDSE}.

%which indicates the expected order of convergence. 

\begin{table}[H] \renewcommand{\arraystretch}{1.3}
\caption{Convergence for grid Fig.~\ref{grid2} for the advection equation (Sec.~\ref{subs:advection_equation}). The number of isotropic refinements done in each block is given in the first column.}
\label{tab:adv_example_grid2}
\centering \begin{tabular}{c||c||c|c||c|c} \hline
order&\bfseries  & \bfseries $\ell_2$ error & \bfseries conv. rate &\bfseries $\ell_{\infty}$ error & \bfseries conv. rate \\ 
\hline\hline 
2 &0 & $2.7 \cdot 10^{-1}$ & --- & $ 8.5 \cdot 10^{-1}$ & --- \\
&1 & $ 2.0 \cdot 10^{-1}$ & 0.4 & $ 7.0 \cdot 10^{-1}$ & 0.3 \\
&2 & $ 1.3 \cdot 10^{-1}$ & 0.7 & $ 4.9 \cdot 10^{-1}$ & 0.5 \\
&3 & $  5.3 \cdot 10^{-2}$ & 1.3 & $ 2.8 \cdot 10^{-1}$ & 0.8 \\
&4 & $ 1.5 \cdot 10^{-2}$ & 1.8 & $8.8 \cdot 10^{-2}$ & 1.7 \\
&5 & $ 3.7 \cdot 10^{-3}$ & 2.0 & $ 2.2\cdot 10^{-2}$ & 2.0 \\
\hline 
4 &0 & $2.0 \cdot 10^{-1}$ & --- & $ 5.4 \cdot 10^{-1}$ & --- \\
&1 & $ 7.6 \cdot 10^{-2}$ & 1.4 & $ 2.7 \cdot 10^{-1}$ & 1.2 \\
&2 & $ 1.2 \cdot 10^{-2}$ & 2.7 & $ 6.7 \cdot 10^{-2}$ & 2.0 \\
&3 & $ 7.7 \cdot 10^{-4}$ & 3.9 & $ 4.9 \cdot 10^{-3}$ & 3.8 \\
&4 & $ 5.0 \cdot 10^{-5}$ & 4.0 & $ 3.5 \cdot 10^{-4}$ & 3.8 \\
&5 & $ 3.7 \cdot 10^{-6}$ & 3.7& $ 3.5 \cdot 10^{-5}$ & 3.3 \\
\hline
6 &0 & $ 1.6 \cdot 10^{-1}$ & --- & $ 5.4 \cdot 10^{-1}$ & --- \\
&1 & $ 3.6 \cdot 10^{-2}$ & 2.2 & $ 1.4 \cdot 10^{-1}$ & 2.0 \\
&2 & $ 2.0 \cdot 10^{-3}$ & 4.2 & $ 1.8 \cdot 10^{-2}$ & 2.9 \\
&3 & $ 7.0 \cdot 10^{-5}$ & 4.9 & $ 1.2 \cdot 10^{-3}$ & 3.9 \\
&4 & $ 3.0 \cdot 10^{-6}$ & 4.6 & $ 8.3 \cdot 10^{-5}$ & 3.8 \\
& 5 & $1.9 \cdot 10^{-7}$ & 4.0 & $5.6 \cdot 10^{-6}$ & 3.9\\
\end{tabular} \end{table}

\begin{table}[H] \renewcommand{\arraystretch}{1.3}
\caption{Convergence for grid Fig.~\ref{grid3} for the advection equation (Sec.~\ref{subs:advection_equation}). The number of isotropic refinements done in each block is given in the first column.}
\label{tab:adv_example_grid3}
\centering \begin{tabular}{c||c||c|c||c|c} \hline
order&\bfseries  & \bfseries $\ell_2$ error & \bfseries conv. rate &\bfseries $\ell_{\infty}$ error & \bfseries conv. rate \\ 
\hline\hline 
2 &0 & $ 2.6 \cdot 10^{-1}$ & --- & $8.5 \cdot 10^{-1}$ & --- \\
&1 & $ 2.0 \cdot 10^{-1}$ & 0.4 & $ 7.0 \cdot 10^{-1}$ & 0.3 \\
&2 & $ 1.3 \cdot 10^{-1}$ & 0.7 & $4.9 \cdot 10^{-1}$ & 0.5 \\
&3 & $  5.3 \cdot 10^{-2}$ & 1.3 & $ 2.8 \cdot 10^{-1}$ & 0.8 \\
&4 & $ 1.5 \cdot 10^{-2}$ & 1.8 & $8.8 \cdot 10^{-2}$ & 1.7 \\
&5 & $ 3.7 \cdot 10^{-3}$ & 2.0 & $2.2 \cdot 10^{-2}$ & 2.0 \\
\hline 
4 &0 & $1.9 \cdot 10^{-1}$ & --- & $ 6.4 \cdot 10^{-1}$ & --- \\
&1 & $  7.5 \cdot 10^{-2}$ & 1.4 & $ 2.5 \cdot 10^{-1}$ & 1.3 \\
&2 & $  1.1 \cdot 10^{-2}$ & 2.7 & $ 6.7 \cdot 10^{-2}$ & 1.9 \\
&3 & $ 7.7 \cdot 10^{-4}$ & 3.9 & $ 5.1 \cdot 10^{-3}$ & 3.7 \\
&4 & $ 4.9 \cdot 10^{-5}$ & 4.0 & $ 3.8 \cdot 10^{-4}$ & 3.7 \\
&5 & $ 3.6 \cdot 10^{-6}$ & 3.8 & $ 3.6 \cdot 10^{-5}$ & 3.1 \\
\hline
6 &0 & $ 1.6 \cdot 10^{-1}$ & --- & $ 5.9 \cdot 10^{-1}$ & --- \\
&1 & $ 3.5 \cdot 10^{-2}$ & 2.2 & $ 1.5 \cdot 10^{-1}$ & 2.0 \\
&2 & $ 2.0 \cdot 10^{-3}$ & 4.2 & $ 2.0 \cdot 10^{-2}$ & 2.9 \\
&3 & $ 8.8 \cdot 10^{-5}$ & 4.5 & $ 1.6 \cdot 10^{-3}$ & 3.6 \\
&4 & $ 3.7 \cdot 10^{-6}$ & 4.6 & $ 1.4 \cdot 10^{-4}$ & 3.8 \\
&5 & $ 2.0 \cdot 10^{-7}$ & 4.2  & $6.8 \cdot 10^{-6}$ & 4.3\\
\end{tabular} \end{table}

\begin{table}[H] \renewcommand{\arraystretch}{1.3}
\caption{Convergence for grid Fig.~\ref{grid4} for the advection equation (Sec.~\ref{subs:advection_equation}). The number of isotropic refinements done in each block is given in the first column.}
\label{tab:adv_example_grid4}
\centering \begin{tabular}{c||c||c|c||c|c} \hline
order&\bfseries  & \bfseries $\ell_2$ error & \bfseries conv. rate &\bfseries $\ell_{\infty}$ error & \bfseries conv. rate \\ 
\hline\hline 
2 &0 & $ 3.5 \cdot 10^{-1}$ & --- & $ 7.2 \cdot 10^{-1}$ & --- \\
&1 & $ 2.1 \cdot 10^{-1}$ & 0.7 & $ 4.9 \cdot 10^{-1}$ & 0.6 \\
&2 & $ 8.7 \cdot 10^{-2}$ & 1.3 & $ 2.8 \cdot 10^{-1}$ & 0.8 \\
&3 & $  2.4 \cdot 10^{-2}$ & 1.8 & $ 8.8 \cdot 10^{-2}$ & 1.7 \\
&4 & $ 6.0 \cdot 10^{-3}$ & 2.0 & $ 2.2\cdot 10^{-2}$ & 2.0 \\
&5 & $ 1.5 \cdot 10^{-3}$ & 2.0 & $5.5\cdot 10^{-3}$ & 2.0 \\
\hline 
4 &0 & $1.4 \cdot 10^{-1}$ & --- & $ 2.7 \cdot 10^{-1}$ & --- \\
&1 & $  2.2 \cdot 10^{-2}$ & 2.7 & $ 7.8 \cdot 10^{-2}$ & 1.8 \\
&2 & $  1.5 \cdot 10^{-3}$ & 3.8 & $ 9.8 \cdot 10^{-3}$ & 3.0 \\
&3 & $ 1.0 \cdot 10^{-4}$ & 3.9 & $ 1.3 \cdot 10^{-3}$ & 3.0 \\
&4 & $ 7.7 \cdot 10^{-6}$ & 3.7 & $ 1.9 \cdot 10^{-4}$ & 2.7 \\
&5 & $ 8.2 \cdot 10^{-7}$ & 3.2 & $ 3.6 \cdot 10^{-5}$ & 2.4 \\
\hline
6 &0 & $ 6.8 \cdot 10^{-2}$ & --- & $ 2.0 \cdot 10^{-1}$ & --- \\
&1 & $ 6.2 \cdot 10^{-3}$ & 3.5 & $ 4.0 \cdot 10^{-2}$ & 2.3 \\
&2 & $ 3.6 \cdot 10^{-4}$ & 4.1 & $ 4.9 \cdot 10^{-3}$ & 3.1 \\
&3 & $ 2.1 \cdot 10^{-5}$ & 4.1 & $ 6.3 \cdot 10^{-4}$ & 3.0 \\
&4 & $ 1.1 \cdot 10^{-6}$ & 4.2 & $ 4.2 \cdot 10^{-5}$ & 3.9 \\
&5 & $ 5.5 \cdot 10^{-8}$ & 4.3 & $2.6 \cdot 10^{-6}$ & 4.0\\
\end{tabular} \end{table}

\section{Adaptivity}\label{sec:adaptivity}

%Until now, we have considered the question how to organize an adaptive mesh and how to discretize the appearing special cases in the grid. 
In this section, we discuss how to automatically generate and evolve problem-dependent meshes. Firstly, we explain how we estimate the error in order to have a measure of the quality of a grid. Secondly, for the advection equation, we compare the quality of a solution on a grid that includes SBP-FD junctions with a simple solution that closes each block with an SBP-SAT interface. Finally, for the Schr{\"o}dinger equation, we show two examples of the solution of a quantum harmonic oscillator, one on a manually created grid and one on an automatically generated mesh.

\subsection{Error estimation}\label{sec:error_estimate}

%To design grids that are adapted to the shape of the solution, 
To adapt a grid to the shape of a solution, we need a measure of the error for a given mesh. For simplicity and efficiency, we estimate the error at the grid points of the present mesh. Note that this can be problematic when the solutions are highly oscillatory. In order to make sure we do not fail to capture oscillations that occur on a scale that is not resolved by the given mesh, special care has to be taken when choosing the initial mesh. If the initial mesh fully resolves all the oscillations and we update the mesh in sufficiently close intervals, we will be able to capture emerging oscillations. 

We base the error estimate on the residual. In case of an approximate solution that is defined continuously on the whole domain, the residual is defined as the difference between the continuous differential operator and the approximate difference operator applied to the approximate solution. This is possible for finite element approximations where a representation of the approximate solution based on some basis functions is known (see \cite{Kormann12}). In our finite difference setting, we instead use a better approximation of the operator as a reference.  

We compute the residual at point $\x_j \in X$ as
\begin{equation}
	R(\x_j,t) =  - \sum_{i=1}^d \underbrace{\left((A_i^{2p} \V)_j - (A_i^{2p+2}\V)_j)\right)}_{:= R_i(\x_j,t)},
\end{equation}
where $A_i^{\star}$ denotes an approximation of the (scaled) derivative operator of the order in the superscript ($\star$) in dimension $i$ (one-sided at all the block boundaries), $\V$ the fully discrete solution, and $2p$ is the order of the inner finite difference stencils used in the simulation. Note that terms including the value (instead of derivatives) of $U$ may occur when we have a potential operator, $V$, in the Schr{\"o}dinger equation. However, we do not get any contribution from those terms since the application of the potential operator is --- seen pointwise --- done without error.

Now we want to use the residual error estimator to decide where to refine the grid. For this purpose we compute the residual block-wise, estimating the block error using the one-sided SBP stencils we have at hand. Observe that in order to estimate the residual error in a block we only need information that is immediately available within that block, with no need for data resident in other blocks. This is an important aspect in a large scale parallel implementation, since such data dependencies would require communication between processors. In order to be able to decide locally for each block whether or not it should be refined, we use a weighted threshold based on the $\ell_2$ norm of the residual on each block. Given a global tolerance that the $\ell_2$ norm of the error on the whole domain shall meet, we allow for a block-wise error according to the block's fraction of the total solution volume. However, the block-wise computed $\ell_2$ norm gives a pessimistic estimate for how the error in the derivatives affects the error in the time propagation. The reason is that we only use one-sided stencils at some block boundaries and in a propagation one can gain up to one (for first derivatives) or two (for second derivatives) orders of accuracy compared to the order of the discretization at the boundaries. Therefore, we scale the error down by a factor $\mathrm{vol}(block)^{q/d}$ at the points where one-sided differences are applied, where $q=1$ for first derivatives and $q=2$ for second derivatives. Moreover, recall that we do not want to refine the blocks isotropically. Instead, we always refine in the direction where the error is the largest. For this purpose, we use $R_i(\x_j,t)$ to estimate the error in dimension $i$.

%Now, we want to use the residual error estimator to decide where to refine the grid. For this purpose, we compute the residual block-wise. In principle, we compute the $\ell_2$ norm of the residual on each block. However, this gives a pessimistic estimate for how the error in the second derivatives affects the error in the time propagation. The reason is that we only use one-sided stencils at some block boundaries and it was numerically observed in~\cite{Nissen11} that the order of a propagator based on an SBP-SAT discretization for the Schr\"{o}dinger equation gains two orders of accuracy compared to the order of the discretization at the boundaries. 

So far, we have only studied the residual. In \cite{Kormann10}, it is shown for the Schr{\"o}dinger equation how the residual relates to the error due to spatial discretization after time $T_{\max}$
\begin{equation}
	 \|e(T_{\max})\|_{\ell_2} \leq \int_0^{T_{\max}} \sqrt{\sum_{\x_j \in X} \mathrm{vol}(B_j) |R(\x_j,\tau)|^2} \, \D \tau,
\end{equation}
where $\mathrm{vol}(B_j)$ is the volume of the block that $\x_j$ belongs to. In practice, we further discretize in time as well, to obtain
\begin{equation}
	\|e(T_{\max})\|_{\ell_2} \leq \sum_{k=1}^{N_t} (\Delta t)_k \sqrt{\sum_{\x_j \in X} \mathrm{vol}(B_j) |R(\x_j,t_k)|^2},
\end{equation}
where $N_t$ is the number of time steps and $t_k, (\Delta t)_k$ are the temporal grid points and the temporal grid size, respectively. If we want to meet a certain tolerance for the error at time $T_{\max}$, we have to make sure that the residual in step $k$ is less than that tolerance times $\frac{(\Delta t)_k}{T_{\max}}$. The same estimate applies for the advection equation as long as boundary conditions that ensure time-reversibility and norm conservation are used.

\subsection{Numerical examples}

\subsubsection{Reduced number of grid points with SBP-FD junction}
\label{sec:adv_eq_application}
To illustrate the usefulness of the SBP-FD junction treatment we consider the advection equation on a more complex computational domain. The grid has three different levels of refinement. Figs. \ref{fig:adv_eq_mesh_naive} and  \ref{fig:adv_eq_mesh_SBPFD} illustrate how the meshes are constructed with and without use of the SBP-FD junction technique. We refer to these meshes as the naive mesh and the SBP-FD junction mesh, respectively. To numerically solve the advection equation on the naive mesh with the techniques described in this paper it is necessary to introduce sub-domains and SBP-SAT interfaces as in Fig. \ref{fig:adv_eq_mesh_naive}, whereas the SBP-FD junction mesh has a reduced number of SBP-SAT interfaces, see Fig. \ref{fig:adv_eq_mesh_SBPFD}. Since the local order of accuracy is decreased in the vicinity of SBP-SAT interfaces it should be beneficial to reduce these to a minimum. Also, introducing the extra SBP-SAT interfaces as in the naive mesh necessarily increases the number of grid points on the coarsest level of refinement. For these reasons the use of the SBP-FD junction treatment should yield a more efficient numerical method. We take $a_1 = -2 , a_2 = -1$ in \eqref{eq:advection_equation} and set the incoming characteristics to zero. The initial data  is given as.
\begin{equation*}\begin{aligned}
u_0(x,y) =& \exp(-200((x-0.25)^2 + (y-0.75)^2)) + \\&\exp(-600((x-0.625)^2 + (y-0.325)^2)) +\\& \exp(-1800((x-0.875)^2 + (y-0.0625)^2)).
\end{aligned}\end{equation*}
%as shown in Fig. \ref{fig:adv_eq_initial_data}.
 Note that the initial data is chosen such that the gradient of the solution is larger where the mesh is finer. The numerical solution is propagated in time until $t = 0.12$ using fourth order SBP operators to discretize spatial derivatives and the SAT terms derived in Sec. \ref{sec:first} to couple grid patches together. At exterior boundaries, Dirichlet conditions are imposed using SAT terms at the left and bottom boundaries. The solution is computed on the naive mesh and the SBP-FD junction mesh and for each mesh the total number of grid points required to obtain relative max errors of $1 \cdot 10^{-3},1 \cdot 10^{-2}$ and $1 \cdot 10^{-1}$ is recorded. Table \ref{tab:adv_eq_adaptive_mesh} displays the required number of grid points to achieve the required relative max errors. As can be seen, the use of the SBP-FD junction treatment reduces the required number of grid points by about $20 \%$. For a computational domain with more refinement levels or more than two spatial dimensions this indicates that the SBP-FD technique will reduce the number of grid points even more due to reducing the number of SBP-SAT interfaces.                 
\begin{table}
\caption{Comparison of number of grid points for the naive mesh and the SBP-FD junction mesh for the simulations (Sec.~\ref{sec:adv_eq_application}).}
\begin{tabular}{lllllll} 
\noalign{\smallskip}\hline\noalign{\smallskip}
Relative max error  & No. grid points (SBP-FD junction mesh)& No. grid points (Naive mesh)\\
\noalign{\smallskip}\hline\noalign{\smallskip}
$1 \cdot 10^{-1}$  	&4327&5209\\
$1 \cdot 10^{-2}$    	&16647&20009\\
$1 \cdot 10^{-3}$    	&65287&78409\\
\noalign{\smallskip}\hline
\end{tabular}
\label{tab:adv_eq_adaptive_mesh}
\end{table}

\begin{figure}[!t]
\centering
\subfloat[][]{
	\includegraphics[width=2.2in]{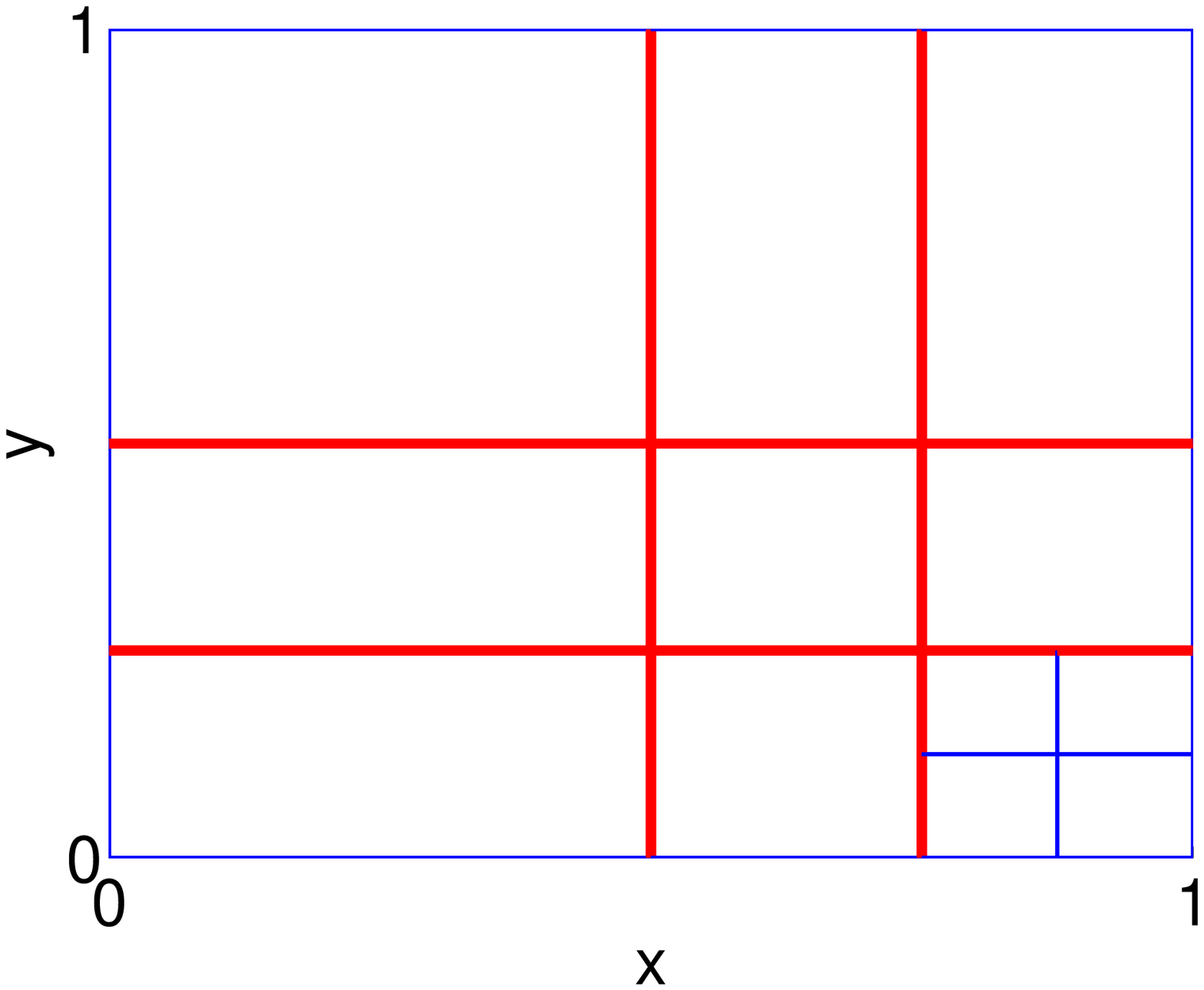}
	\label{fig:adv_eq_mesh_naive}
%\vspace{-90pt}
%\caption{The naive mesh used for the advective equation in Section \ref{sec:adv_eq_application}. Red lines indicate SAT boundaries and blue dotted lines indicate c-FD discretization.}
} \hspace{0.1in}
\subfloat[][]{
	\includegraphics[width=2.2in]{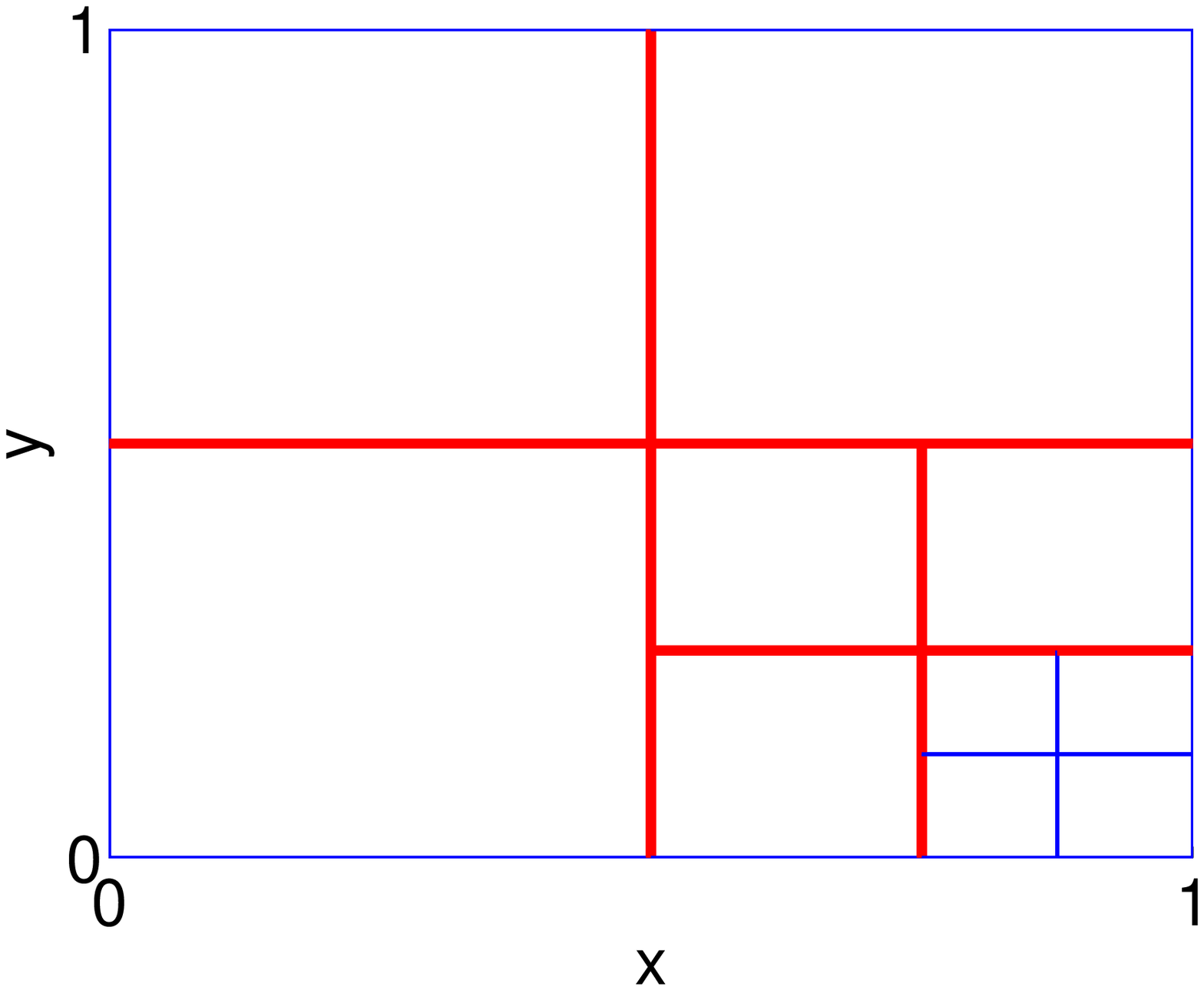}
	\label{fig:adv_eq_mesh_SBPFD}
%\vspace{-90pt}
%\caption{The SBP-FD junction mesh used for the advective equation in Section \ref{sec:adv_eq_application}. Red lines indicate SAT boundaries and blue dotted lines indicate c-FD discretization.}
}
	\caption{Meshes used for the numerical experiments with the advective equation in the experiments in Section~\ref{sec:adv_eq_application}. (a) Naive mesh. (b) SBP-FD junction mesh.}

\end{figure}

%\begin{figure}[!t]
%\centering
%\subfloat[][Naive mesh]{ 
%	\includegraphics[width=2.5in]{naive_mesh} 
%} 
%\subfloat[][SBP-FD junction mesh]{
%	\includegraphics[width=2.5in]{T_mesh}
%}
%\caption{The two different meshes used in Section \ref{sec:adv_eq_application}. Red lines indicates SAT boundaries.}
%\label{fig:adv_eq_mesh}
%\end{figure}

%
%\begin{figure}[!t]
%	\centering
%	\includegraphics[width=4in]{Fig3.eps}
%		\caption{Initial data for the experiment in Section \ref{sec:adv_eq_application}.}
%		\label{fig:adv_eq_initial_data}
%\end{figure}
%

\subsubsection{Error distribution on an adaptive mesh}
\label{sec:ho}
 
Next we study an adaptive discretization of the Schr\"{o}dinger equation. In this example, we consider a quantum harmonic oscillator, described by
\begin{equation}
U_t = \im U_{xx}+\im U_{yy}-\im \left(\frac{m a_x^2}{2} x^2 + \frac{m a_y^2}{2} y^2\right) U,
\end{equation}
with $a_x = a_y = 8$. As initial value we take a Gaussian as given by Eq.~\eqref{Gaussian}, with parameters $\alpha_x = \alpha_y  = 2$, $x_0 = y_0 = 1$, and $k_x = k_y = 0$.
The wave packet has a momentum with a $45 \,^{\circ}$ angle to the coordinate axis and is traveling from $(1,1)$ to $(-1,-1)$. The mesh is illustrated in Fig.~\ref{fig:example3_mesh}. After 1000 time steps with a step size of $4 \cdot 10^{-4}$, the $\ell_2$ error compared to an analytical solution \cite{Tannor07} is~$4.3 \cdot 10^{-6}$.

\begin{figure}[!t]
  \centering
  \subfloat[][]{
	\includegraphics[width=2.2in]{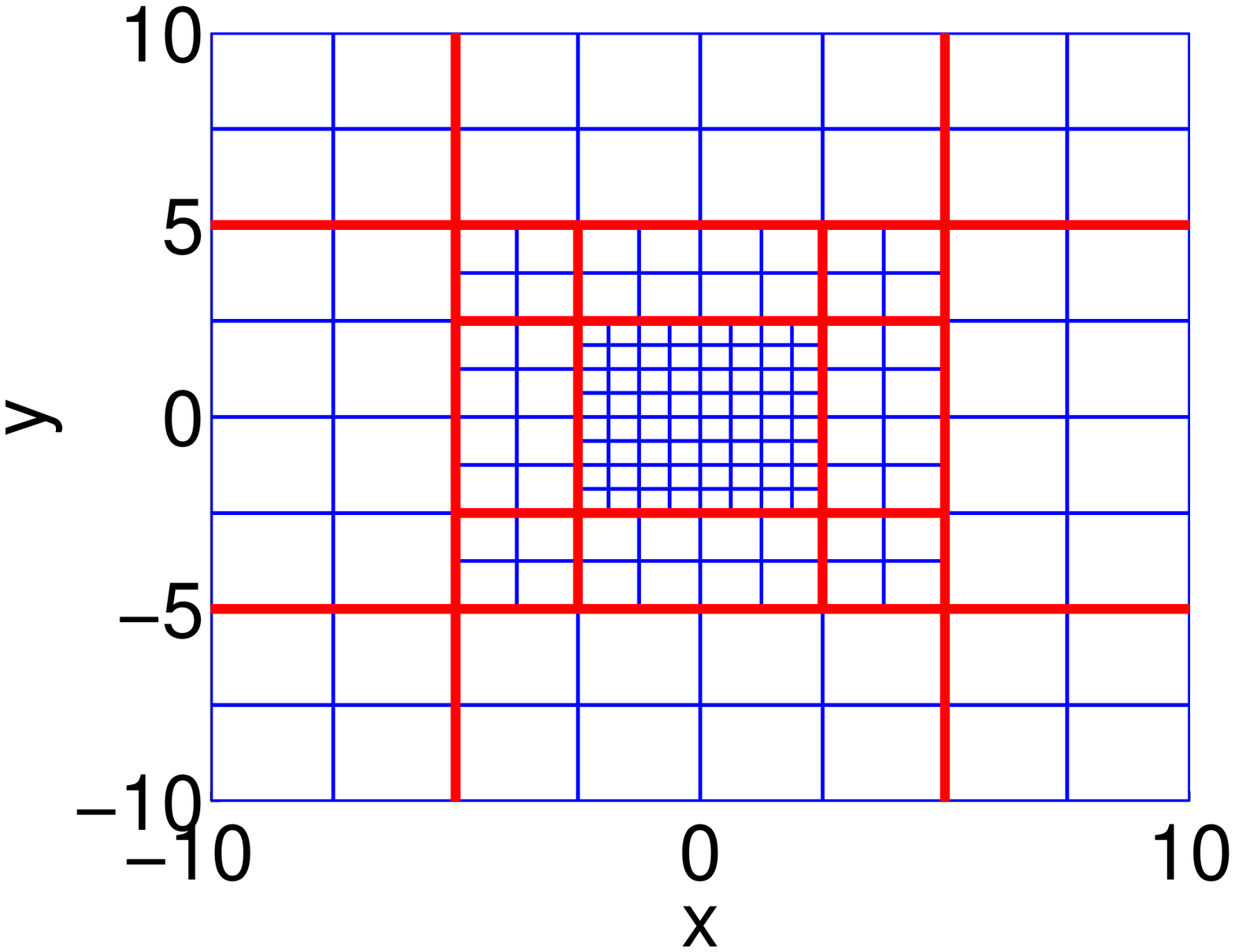}
	\label{fig:example3_mesh}
  } \hspace{0.1in}
  \subfloat[][]{
	\includegraphics[width=2.2in]{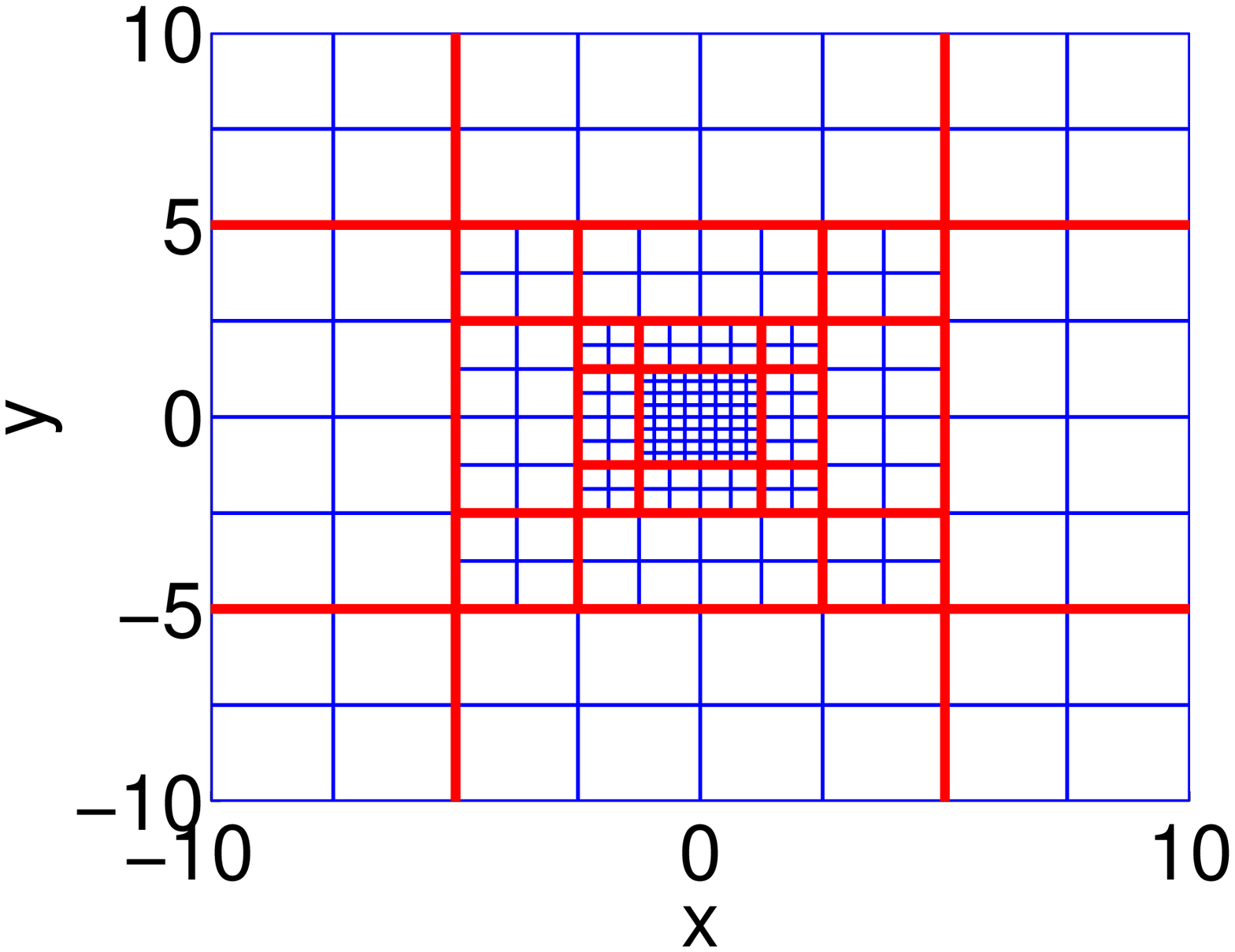}
	\label{fig:example3_mesh_b}
  }
	\caption{Meshes with corners and SBP-FD junctions used in the experiments in Section~\ref{sec:ho}. The mesh in (b) has an additional level of refinement in the four-by-four block region at the center of the mesh in (a).}
	\label{fig:example3}
\end{figure}

\begin{figure}[!t]
\centering
\subfloat[][]{
	\includegraphics[width=2.3in]{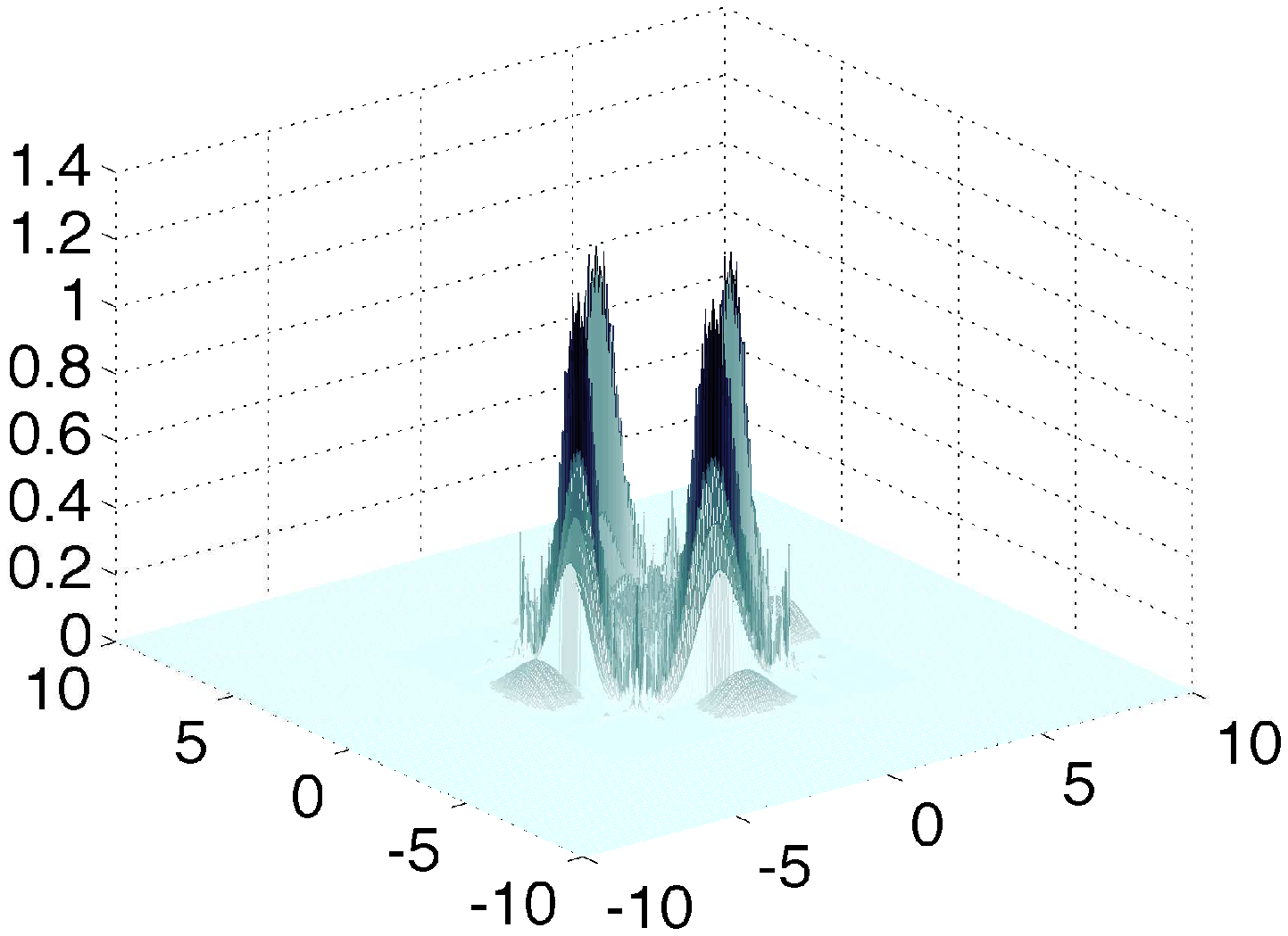}
	\label{fig:error_unsymmetric_nt250}
}
\subfloat[][]{
	\includegraphics[width=2.3in]{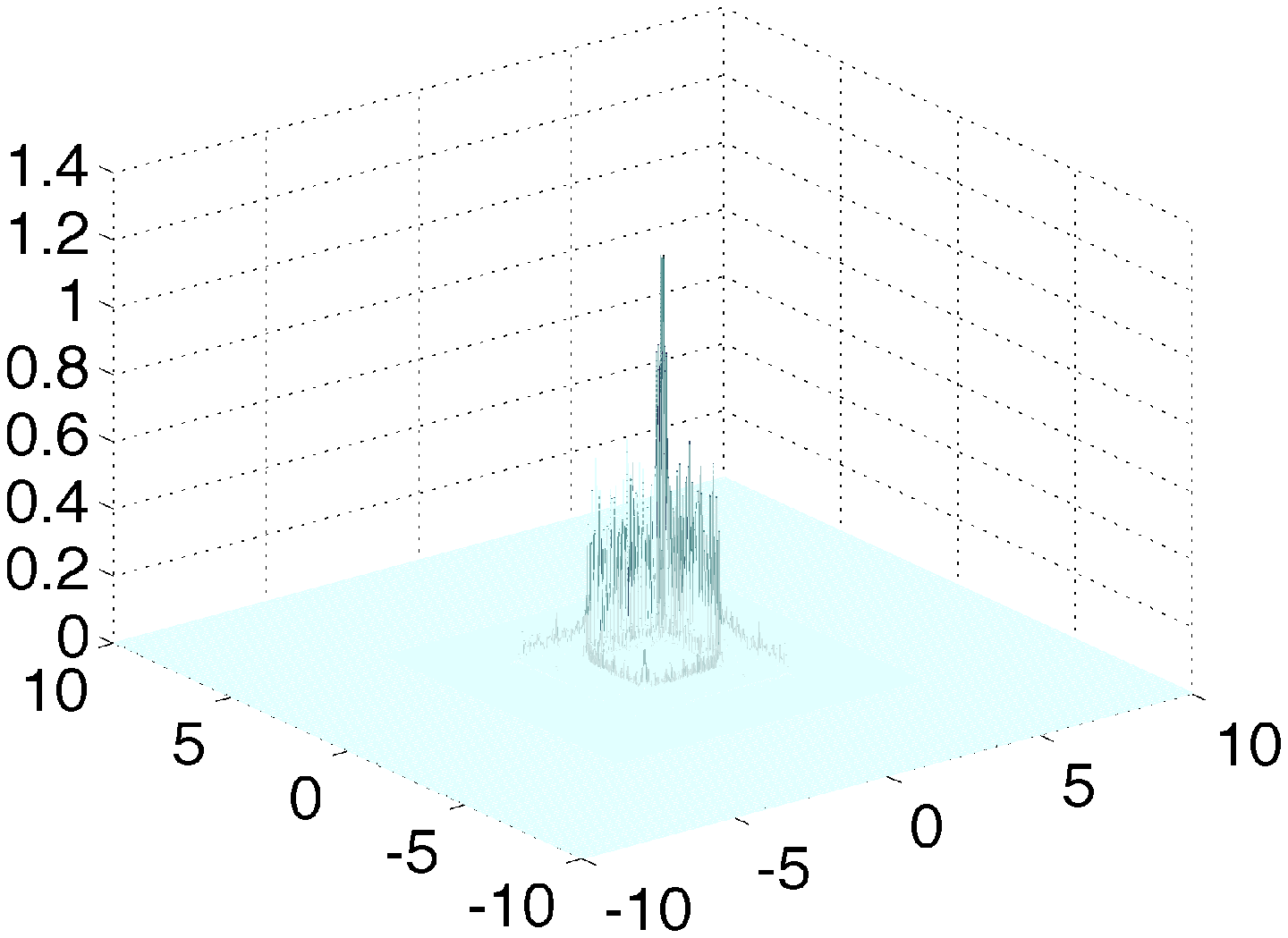}
	\label{fig:error_unsymmetric_nt1000}
}
	\caption{Absolute errors in the solution of a wave packet propagated on the meshes in Fig. \ref{fig:example3}. Observe that the errors are localized to the SBP-SAT interfaces. The error in (a) is on the scale of $10^{-7}$ and the error in (b) is on the scale of $10^{-4}$.}
\end{figure}

Moreover, we have simulated over a shorter time interval, $0.04$, on both the grid from Fig.~\ref{fig:example3_mesh} and the one shown in Fig.~\ref{fig:example3_mesh_b}. The only difference between the two meshes is that in mesh \ref{fig:example3_mesh_b}, the inner most four-by-four blocks are refined once more. The $\ell_2$ error on mesh \ref{fig:example3_mesh_b} is $1.5 \cdot 10^{-5}$ and on mesh \ref{fig:example3_mesh} it is $8.9 \cdot 10^{-8}$. Even though the former mesh is finer, the error is larger. This shows that one should avoid to place SBP-SAT interfaces too close to the center of the solution. Morover, the time steps need to be significantly smaller for the experiments on mesh \ref{fig:example3_mesh_b} than for the ones on mesh \ref{fig:example3_mesh}. The errors in the central region corresponding to meshes \ref{fig:example3_mesh} and \ref{fig:example3_mesh_b} are visualized in Figs. \ref{fig:error_unsymmetric_nt250} and \ref{fig:error_unsymmetric_nt1000}, respectively.

\subsubsection{Error estimation and mesh adaptation}
\label{sec:mesh_design}

Consider the example of a free wave packet. We choose an initial value with parameters $\alpha_x = \alpha_y  = 2$, $k_y = 1$, and $x_0 = y_0 = k_x = 0$ in Eq.~\eqref{Gaussian}. We fix the time step to $10^{-4}$, perform a simulation over 100 time steps, and set a global tolerance of $10^{-5}$. In this case, the local residual scaled by the time step needs to be bounded by $10^{-7}$.
 
\begin{figure}[!t]
\centering
\includegraphics[width=2.2in]{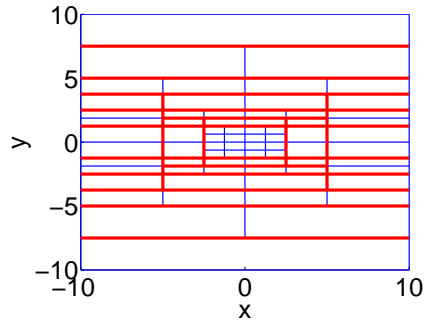}
\caption{Anisotropically refined mesh that is adapted to the wave packet with oscillations in $y$-dimension that is used in the experiment in Section~\ref{sec:mesh_design}.}
\label{fig:example1_mesh}
\end{figure}

First, we generate an initial mesh based on the error estimator described in Section \ref{sec:error_estimate}. Fig.~\ref{fig:example1_mesh} shows the mesh that is adapted to fit our parameters. One can see that the mesh fits well with the shape of the wave function: since the function has oscillations in the $y$-dimension only, the mesh is anisotropic with higher resolution along the $y$-dimension. Computing the Hamiltonian applied to the wave packet $U(\x,0)$ on the automatically generated grid and comparing with the analytical expression of the second derivative, we get an $\ell_2$ error of $2.7 \cdot 10^{-3}$. Hence, the actual error, when the Hamiltonian is applied, is a little larger than the tolerance required for the residual on the mesh. This does not come as a surprise, though, since we have scaled down the comparably large errors at the SBP-SAT interfaces when estimating the error in the mesh. The excerpt of the derivative shown in Fig.~\ref{fig:example1_error2ndder} illustrates the fact that the largest errors are indeed confined to the SBP-SAT interfaces.
 
Next, we propagate in time \textit{without adjusting the grid} for the residual of the propagated wave packet. After the 100 time steps, the error in the solution at time $10^{-2}$ is $1.2 \cdot 10^{-6}$. Hence, the global tolerance is met even though we do not check that the tolerance is met on the generated mesh in each time step. This is possible since the simulation time is rather short such that the solution does not move a lot. If we want to compute over a larger time interval, we would have to readjust the mesh after a number of steps. This experiment shows that the error estimation is effective and maybe a bit pessimistic. In this experiment, we have performed a rather large number of iterations in the Lanczos method so that the temporal error is insignificant. If we instead choose the size of the Krylov space adaptively with the same tolerance, the error becomes $1.3 \cdot 10^{-6}$ instead. Hence, the temporal error is a little smaller than the spatial one, but of similar order.
 
\begin{figure}[!t]
\centering
\includegraphics[width=2.5in]{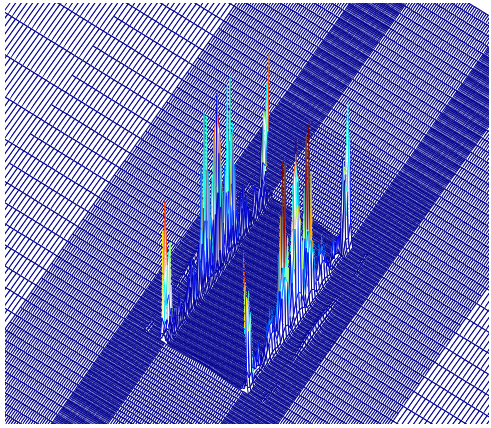}
\caption{Error in the Laplacian of the wave packet on the adaptive grid in Section~\ref{sec:mesh_design}. The figure shows an excerpt from the center of the grid. One can see that the error is centered around the SAT interfaces. The $\ell_\infty$ norm of the error is $6.9 \cdot 10^{-3}$.}
\label{fig:example1_error2ndder}
\end{figure}

\section{Conclusions and outlook}
 \label{section7}

We have presented a prototype implementation of an accurate and stable numerical method for derivative approximation on adaptive meshes. The block adaptivity is organized in a multiblock setting where different blocks can have different refinement levels. Our approach is based on finite differences combined with SBP-SAT interface treatment between neighboring blocks with different levels of refinement.  
The experiments show that SBP-SAT interfaces should be avoided around the maxima of the solution. We have therefore devised an interpolation procedure that treats junctions of different interfaces, allowing for more flexible grid configurations with a minimum number of SBP-SAT interfaces. 
In our setting, the approximation order is reduced at corner points compared to the interior accuracy. For an interior order of $2p$, the stencil order along edges is $p$, and the orders at corner points are $p-1$ and $p-2$ for approximations of first and second derivatives, respectively. However, the global convergence order in time-dependent simulations based on SBP operators is usually higher than the local convergence order at boundaries or interfaces. In simulations we observe an accuracy gain between one and three orders compared to the stencil approximation.
Moreover, we present a strategy to estimate the error on a grid and demonstrate that the error estimator is effective.

\section*{Acknowledgments}

The authors would like to thank Sverker Holmgren and Gunilla Kreiss for valuable insight and discussions. The design of the interpolation operators is based on a Maple sheet by Ken Mattsson. The simulations were performed on resources provided by SNIC-UPPMAX under projects p2003013 and p2005005. 

\begin{appendix}

\section{Interpolation operators at SBP-FD junctions}\label{appendix_operators}

The part of $\left(\begin{array}{c c} 
I_w^{u} \\
I_w^{v} \\
\end{array}\right)$ around the interface is given by for order 4
\renewcommand{\arraystretch}{1.8}
\begin{equation}\begin{aligned}
\notag
\left(\begin{array}{c c} 
\bar I_w^{u} \\
\bar I_w^{v} \\
\end{array}\right)
 =\left(\begin{array}{c c c c c c c c}
 1& 0& 0& 0& 0& 0& 0\\
  0& 1& 0& 0& 0& 0& 0\\
   -\frac{3}{59}& \frac{10}{59}& \frac{48}{59}& \frac{4}{59}& 0 &0& 0\\
        \frac{2}{17} &-\frac{5}{17}& \frac{2}{17}& \frac{20}{17}& -\frac{2}{17}& 0& 0\\
         0& 0& -\frac{2}{17} &\frac{20}{17} &\frac{2}{17}& -\frac{5}{17}& \frac{2}{17}\\
        0& 0 &0 &\frac{4}{59}& \frac{48}{59}& \frac{10}{59}& -\frac{3}{59}\\
         0& 0& 0& 0& 0& 1& 0\\
          0 &0& 0& 0& 0& 0& 1\\
          \end{array}\right),
\end{aligned}\end{equation}
and by for order 6, 

%\begin{equation}\begin{aligned}
%\left(\begin{array}{c c} 
%\bar I_w^{u} \\
%\bar I_w^{v} \\
%\end{array}\right)& =\left(\begin{array}{c c c c c c}
% 1& 0& 0& 0& 0\\
%  0& 1& 0& 0& 0\\
%  0& 0& 1& 0& 0\\
%  \frac{-601}{2711}& \frac{2289}{2711}& -\frac{3117}{2711} & \frac{4320}{2711} & 0\\
%  \frac{1803}{12013} & -\frac{6104}{12013}& \frac{6234}{12013}& 0& \frac{8604}{12013} \\
%  \frac{-3606}{13649}& \frac{11445}{13649}& -\frac{10390}{13649}& 0& \frac{1980}{13649}\\
%  0& 0& 0& 0 &  -\frac{1440}{13649}   \\
%   0& 0& 0& 0 & -\frac{108}{12013}  \\
%  0& 0& 0& 0 & \frac{18}{2711}  \\
%  0& 0& 0& 0& 0\\
%   0& 0& 0& 0& 0\\
%  0& 0& 0& 0& 0\\
%          \end{array}\right. \qquad \ldots \\
%         & \left.\begin{array}{c c c c c c c}
% 0& 0 & 0 & 0 &0 &0\\
%0& 0 & 0 & 0 & 0 & 0\\
%0& 0 & 0 & 0 & 0 & 0\\
% -\frac{198}{2711}& \frac{18}{2711}& 0& 0& 0& 0\\
% \frac{1584}{12013} & -\frac{108}{12013}& 0& 0& 0& 0\\
%   \frac{15660}{13649} & -\frac{1440}{13649}& 0& 0& 0& 0\\
%  \frac{15660}{13649}& \frac{1980}{13649} &-\frac{10390}{13649}& 0&   \frac{11445}{13649}&  -\frac{3606}{13649}    \\
%  \frac{1584}{12013}& \frac{8604}{12013}& \frac{6234}{12013}& 0& -\frac{6104}{12013}& \frac{1803}{12013}&   \\
% -\frac{198}{2711}& 0 & \frac{4320}{2711}&-\frac{3117}{2711}& \frac{2289}{2711} & -\frac{601}{2711}  \\
%  0& 0 & 0 & 1 & 0 & 0\\
%   0& 0 & 0 & 0 &1 &0\\
%  0& 0 & 0 & 0 & 0 & 1\\
%          \end{array}\right). 
% \end{aligned}\end{equation}
 
 \begin{equation}\begin{aligned}
 \notag
\left(\begin{array}{c c} 
\bar I_w^{u} \\
\bar I_w^{v} \\
\end{array}\right)& =\left(\begin{array}{c c c c c c c c c c c c c}
 1& 0& 0& 0& 0 & 0& 0 & 0 & 0 &0 &0\\
  0& 1& 0& 0& 0 & 0& 0 & 0 & 0 &0 &0\\
  0& 0& 1& 0& 0 & 0& 0 & 0 & 0 &0 &0\\
  \frac{-601}{2711}& \frac{2289}{2711}& -\frac{3117}{2711} & \frac{4320}{2711} & 0 &  -\frac{198}{2711}& \frac{18}{2711}& 0& 0& 0& 0\\
  \frac{1803}{12013} & -\frac{6104}{12013}& \frac{6234}{12013}& 0& \frac{8604}{12013} &  \frac{1584}{12013} & -\frac{108}{12013}& 0& 0& 0& 0 \\
  \frac{-3606}{13649}& \frac{11445}{13649}& -\frac{10390}{13649}& 0& \frac{1980}{13649} &  \frac{15660}{13649} & -\frac{1440}{13649}& 0& 0& 0& 0\\
  0& 0& 0& 0 &  -\frac{1440}{13649} &  \frac{15660}{13649}& \frac{1980}{13649} &-\frac{10390}{13649}& 0&   \frac{11445}{13649}&  -\frac{3606}{13649}  \\
   0& 0& 0& 0 & -\frac{108}{12013} &  \frac{1584}{12013}& \frac{8604}{12013}& \frac{6234}{12013}& 0& -\frac{6104}{12013}& \frac{1803}{12013}  \\
  0& 0& 0& 0 & \frac{18}{2711} & -\frac{198}{2711}& 0 & \frac{4320}{2711}&-\frac{3117}{2711}& \frac{2289}{2711} & -\frac{601}{2711}   \\
  0& 0& 0& 0& 0 & 0& 0 & 0 & 1 & 0 & 0 \\
   0& 0& 0& 0& 0 &  0& 0 & 0 & 0 &1 &0\\
  0& 0& 0& 0& 0 &  0& 0 & 0 & 0 & 0 & 1
          \end{array}\right). 
 \end{aligned}\end{equation}
 
$I_{uv}^w$ is then given by 
 \begin{equation}\begin{aligned}
 \notag
I_{uv}^w = \left( \begin{array}{cc} P_{x,w}^{-1}(I_w^u)^T P_{x,u} &  \hspace{5pt}  P_{x,w}^{-1}(I_w^v)^T P_{x,v} \end{array} \right).
\end{aligned}\end{equation}

\end{appendix}

%\begin{acknowledgements}
%If you'd like to thank anyone, place your comments here
%and remove the percent signs.
%\end{acknowledgements}

% BibTeX users please use one of
%\bibliographystyle{spbasic}      % basic style, author-year citations
\bibliographystyle{spmpsci}      % mathematics and physical sciences
%\bibliographystyle{spphys}       % APS-like style for physics
%\bibliography{sbp.bib}   % name your BibTeX data base

%% Non-BibTeX users please use

\end{document}